\documentclass[11pt, reqno]{amsart}
\usepackage{hyperref}
\usepackage{amsmath}
\usepackage{paralist}
\usepackage{amssymb}
\usepackage{amsthm}
\usepackage{amscd}
\usepackage{graphicx,mathrsfs}
\renewcommand{\subjclassname}{2020 Mathematics Subject Classification}
\begin{document}
\title[Cone degenerate p-Laplace equation]
{Existence and uniqueness for cone degenerate p-Laplace equation}
\author[ Hua Chen, Jiangtao Hu, Xiaochun Liu, Yawei Wei, Mengnan Zhang]
{Hua Chen, Jiangtao Hu, Xiaochun Liu, Yawei Wei, Mengnan Zhang}

\address{Hua Chen  \newline
School of Mathematics and Statistics\\ Wuhan University\\ Wuhan 430072, China}
\email{chenhua@whu.edu.cn}

\address{Jiangtao Hu  \newline
School of Mathematical and Sciences\\ Nankai University\\ Tianjin 300071, China}
\email{2120210030@mail.nankai.edu.cn}

\address{Xiaochun Liu \newline
School of Mathematics and Statistics\\  Wuhan University\\ Wuhan 430072, China}
\email{xcliu@whu.edu.cn}

\address{Yawei Wei \newline
School of Mathematical Sciences and LPMC\\ Nankai University\\ Tianjin 300071, China}
\email{weiyawei@nankai.edu.cn}

\address{Mengnan Zhang \newline
School of Mathematical and Sciences \\ Nankai University\\ Tianjin 300071, China}
\email{1120220030@mail.nankai.edu.cn}

\thanks{}

\subjclass[]{}

\begin{abstract}
In this paper, we study the cone degenerate p-Laplace equation. We provide the existence of
the viscosity solutions by proving Alexandrov-Bakelman-Pucci and H\"older estimates. Further more, we give the comparison principle by an equivalent transformation. Finally, we obtain the existence of weak solutions by analyzing the relationship between weak solutions and viscosity solutions.
\end{abstract}

\maketitle
\numberwithin{equation}{section}
\newtheorem{Theorem}{Theorem}[section]
\newtheorem{Remark}{Remark}[section]
\newtheorem{lemma}{Lemma}[section]
 \newtheorem{corollary}{Corollary}[section]
\newtheorem{assumption}{Assumption}[section]
\newtheorem{definition}{Definition}[section]
\newtheorem{Proposition}{Proposition}[section]
\newcommand{\R}{\mathbb{R}}
\newcommand{\B}{\mathbf{B}}
\newcommand{\e}{{\mathcal {E}}}
\newcommand{\M}{{\mathcal{H}}}
\newcommand{\Y}{{\mathbf{P}}}

\renewcommand{\subjclassname}
\allowdisplaybreaks

\section{Introduction}

In this paper, we study the cone degenerate p-Laplace equation.
\begin{equation}\label{eq:11}
t^{-p}div_ \mathbb{B}(|\nabla_ \mathbb{B}u|^{p-2}\nabla_ \mathbb{B}u)
     +t^{-p}(n-p)|\nabla_ \mathbb{B}u|^{p-2}(t\partial_t u)=f(t,x) \ \ \ \ (t,x) \in \mathbb{B},
\end{equation}
and its Dirichlet problem
\begin{equation}\label{H8}
    \begin{cases}
F((t,x),\nabla_\mathbb{B}u,\nabla_\mathbb{B}^2u)=0  \ \ \ & (t,x)\in \mathbb B,\\
u=0  \ & (t,x)\in \partial{\mathbb B}.
\end{cases}
\end{equation}
Here
$$\begin{aligned}
 F((t,x),\nabla_\mathbb{B}u,\nabla_\mathbb{B}^2u)
 &:= t^{-p}div_ \mathbb{B}(|\nabla_ \mathbb{B}u|^{p-2}\nabla_ \mathbb{B}u) \\&+{t^{-p}}(n-p)|\nabla_\mathbb{B}u|^{p-2}(t\partial_t u) -f(t,x),
 \end{aligned}$$
 $f(t,x)$ is continuous, and the domain $\mathbb{B}= (0,1)\times X $ is the stretched  conical domain. 
 The X is a bounded connected open set in $\mathbb{R}^{n-1}$ with smooth boundary, and $x\in $ X.
 Moreover set
\begin{equation}\label{1.2}
\nabla_\mathbb{B}:=(t \partial t,\partial x_1,...,\partial x_{n-1})
 \ \text{and} \ \text{div}_\mathbb{B}:=\nabla_\mathbb{B}\cdot.\end{equation}


The motivation of  nonlinear  equation \eqref{eq:11} comes from the calculus on manifolds with conical singularities. A finite dimensional manifold $B$ with conical singularities is a topological space with a finite subset $$\begin{aligned}B_0= \left\{b_1, \dots, b_M\right\} \subset B
 \end{aligned}$$
 of conical singularities, with the two following properties:

1. $B \backslash B_0$ is a $C^{\infty}$ manifold.

2. Every $b \in B_0$ has an open neighhourhood $U$ in $B$, such that there is a homeomorphism
$$
\varphi: U \rightarrow X^{\Delta},
$$
and $\varphi$ restricts a diffeomorphism
$$
\varphi^{\prime}: U \backslash\{b\} \rightarrow X^{\wedge} .
$$
Here $X$ is  a bounded open subset in $\mathbb{R}^n$ with smooth boundary, and set
$$
X^{\Delta}=\overline{\mathbb{R}}_{+} \times X /(\{0\} \times X) .
$$
This local model is interpreted as a cone with the base $X$. Since the analysis is formulated off the singularity, it makes sense to pass to
$$
X^{\wedge}=\mathbb{R}_{+} \times X,
$$
the open stretched cone with the base $X$. In \eqref{eq:11} we take the simplest case of finite stretched cone such that
$$
\mathbb{B}=(0,1) \times X \text { and }  \partial\mathbb{B}=\{1\}\times X\cup(0,1]\times \partial X.
$$

The Riemannian metric on $X^{\wedge}=\mathbb{R}_{+} \times X$ is
$$
g:=d t^2+t^2 d x^2.
$$
Then we have
$$
g=\left(g_{i j}\right)=\left(\begin{array}{cccc}
1 & & & \\
& t^2 & & \\
& & \ldots & \\
& & & t^2
\end{array}\right) \quad\left(g^{i j}\right)=\left(\begin{array}{cccc}
1 & & & \\
& t^{-2} & & \\
& & \ldots & \\
& & &  t^{-2}
\end{array}\right)
$$
with $G=\sqrt{\left|\operatorname{det}\left(g_{i j}\right)\right|}=t^{n-1}$. Then the gradient with contra-variant components is as the follows,
$$
\nabla_g u=\left(g^{j 1}\frac{\partial u}{\partial t}+\sum_{k=2}^{n} g^{j k} \frac{\partial u}{\partial x_{k-1}}\right)_{j=1, \ldots, n}=\left(\frac{\partial}{\partial t},t^{-2} \frac{\partial}{\partial x_1}, \ldots, t^{-2} \frac{\partial}{\partial x_{n-1}}\right)u.
$$
Moreover, we have
$$
\left|\nabla_g u\right|^2=\left(\nabla_g u, \nabla_g u\right)_g=t^{-2}\left(\left(t \partial_{t} u\right)^2+\left(\partial_{x_1} u\right)^2+\ldots+\left(\partial_{x_{n-1}} u\right)^2\right) .
$$
Now we calculate the p-Laplace operator near to the conical singularity
$$\begin{aligned}
\Delta_{p g}u&= \operatorname{div}_g\left(\left|\nabla_g u\right|^{p-2} \nabla_g u\right)
\\&=  G^{-1} \left( \frac{\partial}{\partial t}+\sum_{j=1}^{n-1} \frac{\partial}{\partial x_j}\right)\left(G\left|\nabla_g u\right|^{p-2}\left( g^{j 1}\frac{\partial u}{\partial t}+\sum_{k=2}^{n} g^{j k} \frac{\partial u}{\partial x_{k-1}}\right)\right).\end{aligned}
$$
By defined \eqref{1.2} it follows that
$$\begin{aligned}
&\left|\nabla_g u\right|^{p-2}\\=&t^{-(p-2)}\left(\left(t \partial_{t} u\right)^2+\left(\partial_{x_1} u\right)^2+\ldots+\left(\partial_{x_{n-1}} u\right)^2\right)^{\frac{p-2}{2}}\\:=&t^{-(p-2)}\left|\nabla_{\mathbb{B}} u\right|^{p-2},
\end{aligned}$$
then we have
$$
\begin{aligned}
&\Delta_{p g} u\\
= & \operatorname{div}_g\left(\left|\nabla_g u\right|^{p-2} \nabla_g u\right) \\
= &t^{-(n-1)}\left(\frac{\partial}{\partial t}\left(t^{n-1} t^{-(p-2)}\left|\nabla_{\mathbb{B}} u\right|^{p-2} \frac{\partial u}{\partial t}\right)\right) \\
& +t^{-(n-1)}\left(\frac{\partial}{\partial x_1}\left(t^{n-1} t^{-(p-2)}\left|\nabla_{\mathbb{B}} u\right|^{p-2} t^{-2} \frac{\partial u}{\partial x_1}\right)\right) \\
& +\ldots+t^{-(n-1)}\left(\frac{\partial}{\partial x_{n-1}}\left(t^{n-1} t^{-(p-2)}\left|\nabla_{\mathbb{B}} u\right|^{p-2} t^{-2} \frac{\partial u}{\partial x_{n-1}}\right)\right) \\
= & t^{-p}\left(\left(t \frac{\partial}{\partial t}\right)\left(\left|\nabla_{\mathbb{B}} u\right|^{p-2}\left(t \frac{\partial u}{\partial t}\right)\right)\right)+t^{-p}(n-p)\left|\nabla_{\mathbb{B}} u\right|^{p-2}\left(t \frac{\partial u}{\partial t}\right) \\
& +t^{-p}\left(\frac{\partial}{\partial x_1}\left(\left|\nabla_{\mathbb{B}} u\right|^{p-2} \frac{\partial u}{\partial x_1}\right)\right)+\ldots+t^{-p}\left(\frac{\partial}{\partial x_{n-1}}\left(\left|\nabla_{\mathbb{B}} u\right|^{p-2} \frac{\partial u}{\partial x_{n-1}}\right)\right) \\
=: & t^{-p} \operatorname{div}_{\mathbb{B}}\left(\left|\nabla_{\mathbb{B}} u\right|^{p-2} \nabla_{\mathbb{B}} u\right)+t^{-p}(n-p)\left|\nabla_{\mathbb{B}} u\right|^{p-2}\left(t\frac{\partial u}{\partial t}\right).
\end{aligned}
$$

 Viscosity solutions were introduced by M.G. Crandall and P.L. Lions \cite{MP} in the case of Hamilton-Jacobi equations and extended to the case of second order elliptic equations by the authors  in \cite{HP2} and  \cite{LMMA}.

  Existence of viscosity solution based upon the vanishing viscosity method    and approximation arguments was studied in 
  \cite{JP}, \cite{P},
  \cite{G4}, \cite{G5}, 
  \cite{HMS}, 
  \cite{HI28}, \cite{HI29}, 
   \cite{MP17}, \cite{MP18}, \cite{MP19}. H.Ishi's in \cite{HI30}  made significant improvements and simplifications to the existence program, 
   in which the
    uniqueness result is generated through the adaptation of the classical Perron's method.
  Afterwards, many articles have adopted Perron's method to study the existence of viscosity solutions,
 such as \cite{MHP}, \cite{HI26},  \cite{HP2}, \cite{RPP},  \cite{IF22}. 
 Among them, authors in \cite{IF22} using the Perron's method
  gave  the existence of viscosity solutions for the  singular fully-nonlinear elliptic equations  modeled on the q-Laplacian when the domain is bounded and has a $\mathcal{C}^2$ boundary. It is worth noting that the singularity mentioned above refers to $1<q<2$, which is completely different from the conical singularity in our article.
%
%
Furthermore, the authors in \cite{GPA} extended the results to the domain whose boundary satisfying a uniform exterior cone condition by studying the Harnack inequality.
The Alexandrov-Bakelman-Pucci (ABP) estimate has been studied widely, and there are many studies about the Harnack inequality  based on the ABP estimate. 
The classical ABP estimate has been obtained regarding the linear Laplace equation with variable coefficients 
in \cite{DN}.
 The ABP estimate was also obtained for the fully-nonlinear elliptic equations  driven by the q-Laplacian,
 in \cite{IF}, \cite{GPF},  \cite{CI}. 
We noticed that the above are studies aimed at 
 a bounded domain. 
On these foundations, the authors on
 \cite{MLA}, \cite{IIV} and \cite{IFA} have generalized the  ABP estimate  to a special unbounded region and obtained the H\"older estimate.
The above work motivates us to explore the  ABP  estimate of the viscosity solutions for \eqref{eq:11}  and H\"older estimate for \eqref{H8}.

The comparison principle  is a crucial step in the Perron's method to prove the existence and uniqueness of solutions. There have been many results in the study of comparison principle, such as \cite{IFA},\ \cite{MHP},\ \cite{HI26},\ \cite{HP2},\ \cite{IF33}.
 H. Ishii \cite{HI26} refined the method of Jensen and provided a comparison principle for fully nonlinear second-order degenerate elliptic equations with strictly growing on zero order term.
 The work in \cite{HP2} and \cite{MHP} extended the conclusion of \cite{HI26}.
   The authors \cite{TA} used the method of equivalent transformation to study the comparison principle of weak solutions. 
This provides us with ideas to deal with the difficulty of missing zero order term of \eqref{eq:11}.

The relationship between viscosity  solutions and weak solutions has been widely studied in recent years. Juutinen, Lindqvist and Manfredi in \cite{PPJ} have already  studied this relation  via the notion of p-harmonic, p-subharmonic and p-superharmonic functions. Afterwards, the authors in \cite{VP} gave a different and simpler proof  by using inf and sup convolutions, and showed that viscosity solutions
 are weak solutions in the case where lower-order term
is continuous and depends only on $(t,x)$. Furthermore, Maria Medina and Pablo Ochoa in \cite{MO}  provides the equivalence between weak solutions and viscosity solutions for  more
general second-order differential equations.
Our main work is to explore ABP estimate of the viscosity solutions for \eqref{eq:11} and H\"older estimate  for $\eqref{H8}$, then further  obtain the existence of
the viscosity solutions to the problem $\eqref{H8}$. In addition, we also provide  the comparison principle of viscosity solutions to \eqref{eq:11}. Finally, we investigate the relationship between  viscosity solutions and weak solutions of \eqref{eq:11}, and provide the existence of weak solutions.


Now, we will provide some notations and the assumptions for this article.
\begin{itemize}\label{notation}
\item[$B_d$]: the open ball with radius $d$.
\item[$\Omega_d$]: the open ``ball'' with radius $d$ on cone.
\item[$d_{(t,x)}$]:  the distance from $(t,x)$ to $\partial \mathbb{B}$.

 \end{itemize}
\begin{assumption}
    We call that the local stretched cone $\mathbb{B}$ satisfies the condition $(G^d_{\mathbb{B}})$ if
there exists a real number $\sigma,$ such that for any $(t,x)\in \mathbb{B}$ there exists a n-dimensional ``ball'' $ \Omega_{ \tilde {R}_{(t,x)}} $ of radius $ \tilde{R}_{(t,x)}\leq K_0d_{(t,x)}$  and $d_{(t,x)}\leq d_0$ satisfying
$$(t,x)\in \Omega_{\tilde{R}_{(t,x)}}\ \ \text{and} \ \ \ \vert \Omega_{\tilde{R}_{(t,x)}}  \backslash \mathbb{B}_{\tilde{R}_{(t,x)}} \vert \geq \sigma \vert \Omega_{\tilde{R}_{(t,x)}}\vert, $$
where $$d_{(t,x)}=\min\{ d\left((t,x),(s,y)\right)\ \big{|}\ \mbox{for any}\ (s,y)\in \partial \mathbb{B}\}
$$
 and $ \mathbb{B}_{\tilde{R}_{(t,x)}} $ is the $($connected$)$ component of  $\mathbb{B}\cap \Omega_{\Tilde{R}_{(t,x)}}$. 
 \end{assumption}
 Clearly, in this article, $\mathbb{B}$ satisfies the condition $(G^d_{\mathbb{B}})$.
 \begin{assumption}\label{assumption2.2}
 There exists  a sequence $\{H_{j}\}$ where each $H_j$ satisfies $(G^d_{H_{j}})$ and has a $\mathcal{C}^2$ boundary such that
$$H_{j}\subset\subset H_{j+1}\subset\subset \mathbb{B} \ \ \ \text{and } \ \ \ \cup_{j}H_{j} =\mathbb{B}.$$
 \end{assumption}

 The main results in this paper are as follows.
 \begin{definition}
 We call that $u\in USC(\Omega)$  $(resp.u\in LSC(\Omega))$ if $u$ is a upper $(resp.lower)$ semicontinuous function defined on ${\Omega}$.
 \end{definition}
\begin{Theorem}\label{T3}
    Let $v\in USC(\mathbb {\overline{B}})$ be a  viscosity subsolution of \eqref{eq:11}, and $\mathop{\sup  }\limits_{\mathbb{{B}}}v ^{+}<+ \infty$. If $f \in C( \overline{\mathbb B})\cap L^{\infty}(\mathbb{B})$,  
     then
      \begin{equation}\label{H11}
          \mathop{\sup }\limits_{\mathbb{B}}v^+ \leq\mathop{\sup }\limits_{\partial\mathbb{B}}v^++C(K_0d_0)^{\frac{p}{p-1}}\mathop {\sup }\limits_{ \mathbb{B}} {\vert\vert t^pf^-\vert\vert^{\frac{1}{p-1}}_{L^{\infty}(\Omega_{2\tilde{R}_{(t,x)}}\cap \mathbb{B})}},
      \end{equation}
      where $C $ depends on $n,\ p,\  
      K_0,\ d_0,\ \sigma$.
 \end{Theorem}
 \begin{corollary}\label{Theorem4.2}
     Let $v\in    C(\mathbb{\overline{B}})$ be a viscosity solution of \eqref{eq:11} and $\mathop{\sup }\limits_{\mathbb{{B}}}\vert v \vert < +\infty$. If $f \in C( \overline{\mathbb B})\cap L^{\infty}(\mathbb{B})$, then
     \begin{equation}
         \mathop{\sup }\limits_{\mathbb{B}}\vert v \vert\leq \mathop{\sup }\limits_{\partial \mathbb{B}}\vert v \vert+C(K_0d_0)^{\frac{p}{p-1}}\mathop{\sup}\limits_{\mathbb{B}} {\vert\vert t^pf\vert\vert^{\frac{1}{p-1}}_{L^{\infty}(\Omega_{2\tilde{R}_{(t,x)}}\cap \mathbb{B})}},
     \end{equation}
     where $C$ depends on $n,\ p,\ \sigma,\ K_0,\ d_0$.
 \end{corollary}
\begin{Theorem}\label{C}
      Let $v\in C(\mathbb{\overline{B}})$  be a  viscosity solution of \eqref{H8}, $f \in C( \overline{\mathbb B})\cap L^{\infty}(\mathbb{B})$,  
     then  there exists constant $\alpha_1$ depending on $K_0,\ d_0,\ \sigma,\ n,\ p$ such that, for all $\rho \in (0,\alpha_1]$ the following estimate holds
     \begin{equation}\label{T1.10}
        ||v(t,x)||_{\rho,\overline{\mathbb B}}\leq C\mathop{\sup}\limits_{\mathbb{B}} {\vert\vert t^pf\vert\vert^{\frac{1}{p-1}}_{L^{\infty}(\Omega_{2\tilde{R}_{(t,x)}}\cap \mathbb{B})}},\\
\end{equation}
 where $C $ depends on $n,\ p,\  
      K_0d_0,\ \sigma,\ \rho$.
    Here the h\"older norm is
  \begin{equation}\label{eq:55}
    ||v(t,x)||_{\rho,\overline{\mathbb B}}=||v(t,x)||_{L^{\infty}(\overline{\mathbb B})}+\mathop{\sup}\limits_{\overline{\mathbb{B}}\times \overline{\mathbb{B}}} \frac{|v(t,x)-v(s,y)|}{|(t,x)-(s,y)|^{\rho}}.
    \end{equation}
   \end{Theorem}
   \begin{Theorem}\label{B}
 Let  $f \in C( \overline{\mathbb B})\cap L^{\infty}(\mathbb{B})$, then
     \eqref{H8} has at least one viscosity solution $v$ satisfying \eqref{T1.10}.
    \end{Theorem}
\begin{Theorem}[Comparison principle]\label{A}
    Let $u\in USC(\overline{\mathbb B})$ be  a viscosity subsolution of equation \eqref{eq:11} and $\mathop{\lim}\limits_{t \to 0} \sup u$ exists; let $v\in LSC(\overline{\mathbb B})$ be a viscosity supersolution of equation \eqref{eq:11} and $\mathop{\lim}\limits_{t \to 0} \inf v$ exists. In addition, $\mathop{\lim}\limits_{t \to 0}\sup u\leq \mathop{\lim}\limits_{t \to 0}\inf v$,
    and $v(t,x)$ is bounded from above. If $f\in C(\overline{\mathbb B})$ and there exists constant $\omega>0$ such that $t^pf(t,x)\ge \omega$ for all $(t,x)\in \mathbb B$. Then we can deduce $u\le v$ in $\mathbb B$ when $u\le v$ on $\partial \mathbb B$.
\end{Theorem}

\begin{Remark}\label{Remark 3.1}
From the  proof process of Theorem \ref{A} , we can see that the assumption of $ f$ can be appropriately
weakened. It is easy to see that when $\alpha\to \infty $, there exists a subsequence $\{(t_{\alpha_j},x_{\alpha_j})\}$ converging to a point $(z_0^*,z^*)\in \mathbb{B}$. We only need $f$ to satisfy 
$$\lim \limits_{(t,x)\to (z_0^*,z^*)}t^pf=(z_0^*)^{p}f(z_0^*,z^*)\neq0.$$
\end{Remark}
\begin{Theorem}\label{Theorem 1.5}
If $u\in L^{\infty}(\mathbb{B})$ is a viscosity supersolution to \eqref{eq:11} and $f\in C(\mathbb{B})$, then  $u$ is a  weak supersolution to \eqref{eq:11}, and $u\in\mathbb{H}_{p}^{1,\frac{n-p}{p}}(\mathbb{B})$.
\end{Theorem}
\begin{Theorem}\label{Theorem 1.6}
If $u\in L^{\infty}(\mathbb{B})$ is a viscosity solution to \eqref{eq:11} and $f\in C(\mathbb{B})$, then  $u$ is a  weak solution to \eqref{eq:11}.
\end{Theorem}

We highlight 
that the degenerate p-Laplace equation studied in this paper is a generalization and further development of classic non-degenerate p-Laplace equations, and other existing degenerate p-Laplace equations. Firstly, our model that the cone degenerate p-Laplace equation is new, comparing to the problems studied in \cite{CI,IF,IFA,IIV}. Secondly, the characteristic and innovation of this paper lie in the fact that the gradient operator 
itself is degenerate, which originates from its geometric background. Thirdly, the absence of  zero order term and the cone type degeneracy in this equation pose difficulties in the analysis process. Finally, we  solve these problems by establishing a  new distance space and structuring a special equivalent transformation.

 This paper is organized as follows. In Section 2, We mainly do some preparatory works, presenting some definitions and remarks.  In Section 3, we give the proofs of  Theorem \ref{T3} and Theorem \ref{C} to obtain the ABP and H\"older estimates.
  Furthermore, we provide a proof of Theorem \ref{B} and obtain the existence of  solutions for \eqref{H8}.
  In Section 4, we prove  Theorem \ref{A} to obtain the comparison principle of the viscosity solutions.
 In Section 5, we prove  Theorem \ref{Theorem 1.5} to obtain the  relationship between  viscosity solutions and weak solutions of \eqref{eq:11}.

\section{Preliminaries}
In this chapter, we will make some preparations by providing some definitions and remarks.
In order to express the weak solutions for \eqref{eq:11}, we need
the adequate distribution spaces. To define the weighted Sobolev spaces on the stretched cone $\mathbb{B}$, we first introduce the weighed
Sobolev spaces and weighted $\mathbb{L}^\gamma_p$ spaces on $\R_+^n$.
In addition, we will also provide the definition of viscosity  solutions and other preparatory content.


\begin{definition}\label{def-1}
For the weight data $\gamma \in \R$ and $(t,x)\in \mathbb{B}$, we say that $u(t,x)\in
\mathbb{L}_p^\gamma(\mathbb{B})$ if $u\in\mathcal{D}^\prime(\mathbb{B})$ and
$$\| u\|_{\mathbb{L}_p^\gamma(\mathbb{B})}=\big(\int_{\mathbb{B}}
|{\left(t(1-t)d(x,\partial X)\right)}^{\frac{n}{p}-\gamma}u(t,x)|^p \frac{dt}{t}dx\big)^{\frac{1}{p}}<
+\infty.$$
\end{definition}
\begin{definition}
For $m\in \mathbb{N}$, and $\gamma\in \mathbb{R}$, the spaces
$$\mathbb{H}^{m,\gamma}_p(\mathbb{B}):=\Big\{u\in\mathcal{D}^\prime(\mathbb{B}):
 (t\partial_{t})^\alpha\partial_{x}^\beta u \in \mathbb{L}_p^\gamma(\mathbb{B})\Big\},$$ for arbitrary $\alpha\in
\mathbb{N}$, $\beta\in \mathbb{N}^{\mathbb{N}-1}$, and $\alpha+|\beta|\leq m$.
 Denote $\mathbb{H}^{m,\gamma}_{p,0}(\mathbb{B})$ as the subspace  $\mathbb{H}^{m,\gamma}_p(\mathbb{B})$ which is defined as the closure of $C_{0}^{\infty}(\mathbb{B})$ with  respect to the norm $||\cdot||_{\mathbb{H}^{m,\gamma}_p(\mathbb{B})}$.
\end{definition}
\begin{Remark}
$\mathbb{H}^{m,\gamma}_p(\mathbb{B})$ is a Banach space.
\end{Remark}
We only state completeness here. Let $\{u_j\}$ is a
arbitrary Cauchy sequence in $\mathbb{H}^{m,\gamma}_p(\mathbb{B})$, then $\{u_j\}$  and   $\{(t\partial_{t})^\alpha\partial_{x}^\beta u_j \}$ are  Cauchy sequences in  $\mathbb{L}_p^\gamma(\mathbb{B})$.
By Definition \ref{def-1}, ${\left(t(1-t)d(x,\partial X)\right)}^{\frac{n}{p}-\gamma}u_j t^{-\frac{1}{p}}$ is a Cauchy sequence in $L^{p}(\mathbb{B})$, then $${\left(t(1-t)d(x,\partial X)\right)}^{\frac{n}{p}-\gamma}u_j t^{-\frac{1}{p}}\to \varphi $$ in $L^{p}(\mathbb{B})$. Then we have $u_j \to u$ in $\mathbb{L}_p^\gamma(\mathbb{B})$ with $$u=\varphi \times \left({\left(t(1-t)d(x,\partial X)\right)}^{\frac{n}{p}-\gamma} t^{-\frac{1}{p}}\right)^{-1}.$$
Similarly $(t\partial_{t})^\alpha\partial_{x}^\beta u_j \to v$ in $\mathbb{L}_p^\gamma(\mathbb{B})$.

For all $\psi\in C_{0}^{\infty}(\mathbb{B})$, $\alpha\in
\mathbb{N}$, $\beta\in \mathbb{N}^{\mathbb{N}-1}$, and $\alpha+|\beta|\leq m$ , $$\int_{\mathbb{B}}u_j (t\partial_{t})^\alpha\partial_{x}^\beta \psi \frac{dt}{t}dx={-1}^{\alpha+|\beta|}\int_{\mathbb{B}} (t\partial_{t})^\alpha\partial_{x}^\beta u_j  \psi\frac{dt}{t}dx.$$
 Since
$$\begin{aligned}
&|\int_{\mathbb{B}}(u_j-u)(t\partial_{t})^\alpha\partial_{x}^\beta\psi \frac{dt}{t}dx|\\
\leq & \int_{\mathbb{B}}|u_j-u|^{p}{\left(t(1-t)d(x,\partial X)\right)}^{(\frac{n}{p}-\gamma)p}\frac{dt}{t}dx\\ \cdot & \int_{\mathbb{B}} |(t\partial_{t})^\alpha\partial_{x}^\beta\psi| ^{q}{\left(t(1-t)d(x,\partial X)\right)}^{-(\frac{n}{p}-\gamma)q}\frac{dt}{t}dx,
  \end{aligned}$$
then $$\int_{\mathbb{B}}u_j (t\partial_{t})^\alpha\partial_{x}^\beta \psi \frac{dt}{t}dx \to \int_{\mathbb{B}}u (t\partial_{t})^\alpha\partial_{x}^\beta \psi \frac{dt}{t}dx . $$
Similarly,
 $$\int_{\mathbb{B}}(t\partial_{t})^\alpha\partial_{x}^\beta u_j \psi \frac{dt}{t}dx \to \int_{\mathbb{B}}v \psi \frac{dt}{t}dx ,$$
 then $$\int_{\mathbb{B}}u (t\partial_{t})^\alpha\partial_{x}^\beta \psi \frac{dt}{t}dx ={-1}^{\alpha+|\beta|}\int_{\mathbb{B}}v \psi \frac{dt}{t}dx,$$
 that is, $ (t\partial_{t})^\alpha\partial_{x}^\beta u_j = v$.
 So we have $u_j\to u $  in $\mathbb{H}^{m,\gamma}_p(\mathbb{B})$.

Through simple calculations, we get the divergence form of \eqref{eq:11}
$$ div_ \mathbb{B}(t^{-p}|\nabla_ \mathbb{B}u|^{p-2}\nabla_ \mathbb{B}u)
    +nt^{-p}|\nabla_ \mathbb{B}u|^{p-2}(t\partial_t u)=f(t,x),$$
    then we can define the weak supersolution of \eqref{eq:11} as following.
\begin{definition}[weak solution] \label{D2}A function $ u \in \mathbb{H}^{1,\frac{n-p}{p}}_p(\mathbb{B})$ is a weak supersolution (resp. subsolution) to \eqref{eq:11} if
\begin{equation}\label{D2.2}
\int \limits_{\mathbb{B}}t^{-p}|\nabla_ \mathbb{B}u|^{p-2}\nabla_ \mathbb{B}u\cdot \nabla_ \mathbb{B}\psi \frac{dt}{t}dx\geq (resp.\leq)\int \limits_{\mathbb{B}}(-f+nt^{-p}|\nabla_ \mathbb{B}u|^{p-2}(t\partial _{t} u) )\psi \frac{dt}{t}dx\end{equation}
for all non-negative $\psi\in \mathcal{C}_{0}^{\infty}(\mathbb{B}).$
\end{definition}
Next, we provide an equivalent definition of weak solutions, and provide a brief proof. In the following calculations, we will use the following  definition.
\begin{definition}[weak solution]\label{D3}A function $u\in\mathbb{H}^{m,\gamma}_p(\mathbb{B})$ is a weak supersolution (resp. subsolution) to \eqref{eq:11} if
\begin{equation}\label{D2.3}
\int \limits_{\mathbb{B}}|\nabla_ \mathbb{B}u|^{p-2}\nabla_ \mathbb{B}u\cdot \nabla_ \mathbb{B}\psi \frac{dt}{t}dx\geq (resp.\leq)\int \limits_{\mathbb{B}}(-t^pf+(n-p)|\nabla_ \mathbb{B}u|^{p-2}(t\partial _{t} u) )\psi \frac{dt}{t}dx\end{equation}
for all non-negative $\psi\in \mathcal{C}_{0}^{\infty}(\mathbb{B}).$
\end{definition}
\begin{proof}
Assume $u\in\mathbb{H}^{m,\gamma}_p(\mathbb{B}) $  satisies \eqref{D2.2} for all non-negative $\psi\in \mathcal{C}_{0}^{\infty}(\mathbb{B})$, since $t^{p}\psi\in \mathcal{C}_{0}^{\infty}(\mathbb{B})$, so \eqref{D2.3} holds. Conversely, if $u\in\mathbb{H}^{m,\gamma}_p(\mathbb{B}) $  satisies \eqref{D2.3} for all non-negative $\psi\in \mathcal{C}_{0}^{\infty}(\mathbb{B}),$ since $t^{-p}\psi\in \mathcal{C}_{0}^{\infty}(\mathbb{B})$, so we can get \eqref{D2.2}.
\end{proof}
\begin{definition}[Viscosity solution]\label{D1}

 Suppose  {for any} $A\ge B$,
 $$G((t,x),w,P,A)  \ge G((t,x),w,P,B),$$
  a function $u\in LSC(\mathbb{B})(\text{resp}.u\in USC(\mathbb{B}))$ is called a viscosity supersolution $(\text{resp.subsolution})$
of equation \eqref{H1}
\begin{equation} \label{H1}
    G((t,x),u,\nabla_\mathbb{B}u,\nabla_\mathbb{B}^2u)=0,
\end{equation}
 if for all ${\phi}\in C^2(\mathbb{B})$  the following inequality holds at each local minimum $($resp.
maximum$)$ point $(t_0,x_0)\in \mathbb{B}$ of $u-{\phi}$
$$G((t_0,x_0),u(t_0,x_0),\nabla_\mathbb{B}{\phi}(t_0,x_0),\nabla_\mathbb{B}^2{\phi}(t_0,x_0)) \le (resp. \geq) 0.$$
We call that $u$ is a viscosity solution of \eqref{H1} when it is both a subsolution and a supersolution.
\end{definition}

\begin{Remark}
 Let us verify that when $u$ is classical supersolution of $\eqref{H1}$ $($that is $u\in C^2(\mathbb{B})$ and $G\leq 0$ holds for all $(t,x)\in \mathbb{B})$, the above definition is reasonable. When $u\in C^2(\mathbb{B})$, let $t=e^a$, then $u(t,x)=u(e^a,x)$. We define $\overline u(a,x)=u(e^a,x)$, then $$\nabla\overline u(a,x)=\nabla_\mathbb{B}u(t,x),\nabla^2\overline u(a,x)=\nabla_\mathbb{B}^2 u(t,x).$$ It is indeed because
\begin{equation}
    \begin{cases}\partial_a \overline u=\partial_t u\cdot e^a=t\partial_t u, \\
    \partial_x \overline u=\partial_x u,
    \end {cases}
\end{equation}
\begin{equation}
        \begin {cases}
\partial^2_a \overline u=\partial_a (t\partial_t u)=t\partial_t(t \partial_t u)=(t\partial_t)^2u, \\
\partial_a(\partial_x \overline u)=\partial_a(\partial_x u)=t\partial_t(\partial_x u), \\
\partial_x(\partial_a \overline u)=\partial_x(t\partial_t u)=\partial_x(t\partial_t u), \\
\partial^2_x \overline u=\partial^2_x u,
\end {cases}	
\end{equation}
where
$\nabla_\mathbb{B}^2 u(t,x)=\begin{pmatrix}(t\partial_t)^2u & (t\partial_t)\partial_xu \\\partial_x(t\partial_t)u & \partial_x^2u\end{pmatrix}$.

Given $(t_0,x_0) \in \mathbb{B},\ \phi \in C^2(\mathbb{B}),$ $u-\phi$ attains local minimum at $(t_0,x_0)$. Let $\ln t_0=a_0$, 
  then $u(e^a,x)-\phi(e^a,x)$ attains local minimum at $(a_0,x_0)$ and $(\overline u-\overline \phi)(a,x)$ attains local minimum at $(a_0,x_0)$.
  So we can get $$\nabla \overline u(a_0,x_0)=\nabla \overline \phi(a_0,x_0),\ \nabla^2\overline u(a_0,x_0)\ge \nabla^2\overline \phi(a_0,x_0),$$
   that is, $$\nabla_\mathbb{B} u(t_0,x_0)=\nabla_\mathbb{B} \phi(t_0,x_0),\ \nabla^2_\mathbb{B} u(t_0,x_0)\ge \nabla^2_\mathbb{B} \phi(t_0,x_0).$$ We know that $$G((t_0,x_0),u(t_0,x_0),\nabla_\mathbb{B}u(t_0,x_0),\nabla_\mathbb{B}^2u(t_0,x_0)) \leq 0,$$  so the following inequality holds,
$$G((t_0,x_0),u(t_0,x_0),\nabla_\mathbb{B}{\phi}(t_0,x_0),\nabla_\mathbb{B}^2{\phi}(t_0,x_0)) \le0.$$
\end{Remark}
\begin{definition}[Pucci operators]\label{D11}
    \begin{equation}
         \mathcal{M}^+_{\lambda,\Lambda}(X)=\sup \limits_{\lambda I\le A\le \Lambda I} tr(AX)=\Lambda \sum \limits_{e_i>0} e_i+\lambda \sum\limits_{e_i<0} e_i,
         \end{equation}
         \begin{equation}
        \mathcal{M}^-_{\lambda,\Lambda}(X)=\inf \limits_{\lambda I\le A\le \Lambda I} tr(AX)=\lambda \sum \limits_{e_i>0} e_i+\Lambda \sum\limits_{e_i<0} e_i,
        \end{equation}
where $X$ is matrix and $\{e_i\}$ are eigenvalues of $X$.
\end{definition}
\begin{definition}[Distance on cone \cite{CLT}]\label{cone 9}
Since the distance on cone is $ds^2=\frac{1}{t^2}(
dt)^2+\Sigma_{i=1}^{n-1}(dx_i)^2$, we obtain the distance between point $z=(t,x_1,\cdot\cdot\cdot,x_{n-1})$ and $z_0=(t_0,x_1^0,\cdot\cdot\cdot,x_{n-1}^0)$ on cone is
 \begin{equation}
  d(z,z_0)=\sqrt{(\ln t-\ln t_0)^2+\Sigma_{i=1}^{n-1}(x_i-x_i^0)^2},
  \end{equation}
and we define $|z_0-z|=d(z_0,z)$.

For simplicity, we introduce the open ``ball'' in $ \mathbb{R}_{+}^{n} $ in the sense of measure $\frac{dt}{t}dx$ with center $w=(s,y)=(s,y_1,\cdots,y_{n-1})\in \mathbb{R}_{+}^{n}$ and radius r as follows.
\begin{equation}
\Omega_r(s,y):= \left\{ (t,x) \in R_+^n\ \big{|}\ (lnt-lns)^2+\Sigma_{i=1}^{n-1}(x_i-y_i)^2<r^2\right\}.
 \end{equation}
\end{definition}
\begin{definition}[upper $\epsilon$-envelope]If u is a real-valued u.s.c. function on $\mathbb{B}$, then we call the upper $\epsilon$-envelope of u on $\mathbb B_{\epsilon}$ by
 \begin{equation}
 u^\epsilon(t,x)=\max\left\{u(s,y)+\sqrt{\epsilon^2-|(e^t,x)-(e^s,y)|^2}\ \big{|}\ |(e^t,x)-(e^s,y)|\leq \epsilon
 \right\},
   \end{equation}
   with
    \begin{equation}\mathbb B_{\epsilon}=\{(t,x)\in  \mathbb B\ \big{|}\ |(e^t,x)-(e^s,y)|^2>\epsilon^2,\ \text{for all}\ (s,y)\in \partial  \mathbb B \cup \{0\}\times \partial X\}.\end{equation}
    \end{definition}


Next, for convenience, we will simplify $\eqref{eq:11}$ taking two-dimensional as an example,
    \begin{align*}
    & div_\mathbb{B}(|\nabla_\mathbb{B}u|^{p-2}\nabla_\mathbb{B}u)
     \\& = t\partial_t(|\nabla_\mathbb{B}u|^{p-2}(t\partial_tu))+\partial_x(|\nabla_\mathbb{B}u|^{p-2}(\partial_xu)) \\
     & =|\nabla_\mathbb{B}u|^{p-2} (t\partial_t)^2u+(p-2)t\partial_tu|\nabla_\mathbb{B}u|^{p-3}\frac{\nabla_\mathbb{B}u}{|\nabla_\mathbb{B}u|}\cdot ((t\partial_t)^2u,(t\partial_t)\partial_xu)\\
      & +|\nabla_\mathbb{B}u|^{p-2} \partial_x^2u+(p-2)\partial_xu|\nabla_\mathbb{B}u|^{p-3}\frac{\nabla_\mathbb{B}u}{|\nabla_\mathbb{B}u|}\cdot (\partial_x(t\partial_t)u,\partial_x^2u ) \\
      &=|\nabla_\mathbb{B}u|^{p-2}tr\begin{pmatrix}
      (t\partial_t)^2u & (t\partial_t)\partial_xu \\
       \partial_x(t\partial_t)u & \partial_x^2u
       \end{pmatrix}\\
     &+(p-2)|\nabla_\mathbb{B}u|^{p-4}t\partial_tu(t\partial_tu\cdot (t\partial_t)^2u+\partial_xu\cdot (t\partial_t)\partial_xu)\\
     &+(p-2)|\nabla_\mathbb{B}u|^{p-4}\partial_xu(t\partial_tu\cdot \partial_x(t\partial_t)u+\partial_xu\cdot \partial_x^2u)\\
     &=|\nabla_\mathbb{B}u|^{p-2}tr\begin{pmatrix}
     (t\partial_t)^2u & (t\partial_t)\partial_xu \\
     \partial_x(t\partial_t)u & \partial_x^2u
     \end{pmatrix}\\
     &+(p-2)|\nabla_\mathbb{B}u|^{p-4}
     tr\begin{pmatrix}
    (t\partial_tu)^2 & t\partial_tu\partial_xu \\
    t\partial_tu\partial_xu & (\partial_xu)^2
    \end{pmatrix}\begin{pmatrix}
    (t\partial_t)^2u & (t\partial_t)\partial_xu \\
    \partial_x(t\partial_t)u & \partial_x^2u
     \end{pmatrix} \\
     &=|\nabla_\mathbb{B}u|^{p-2}tr\begin{pmatrix}
    (t\partial_t)^2u & (t\partial_t)\partial_xu \\
    \partial_x(t\partial_t)u & \partial_x^2u
    \end{pmatrix}\\
     &+(p-2)|\nabla_\mathbb{B}u|^{p-4}
    tr \Bigg(\nabla_\mathbb{B}u^T\nabla_Bu \begin{pmatrix}
    (t\partial_t)^2u & (t\partial_t)\partial_xu \\
    \partial_x(t\partial_t)u & \partial_x^2u
    \end{pmatrix}\Bigg)\\
     &=|\nabla_\mathbb{B}u|^{p-2}tr(\nabla_\mathbb{B}^2u)+(p-2)|\nabla_\mathbb{B}u|^{p-4}tr (\nabla_\mathbb{B}u^T\nabla_\mathbb{B}u\nabla_\mathbb{B}^2u)\\
     &=|\nabla_\mathbb{B}u|^{p-2}tr(\nabla_\mathbb{B}^2u)+(p-2)|\nabla_\mathbb{B}u|^{p-2}tr \Bigg(\frac{\nabla_\mathbb{B}u^T}{|\nabla_\mathbb{B}u|}\frac{\nabla_\mathbb{B}u}{|\nabla_\mathbb{B}u|}\nabla_\mathbb{B}^2u\Bigg)\\
     &=|\nabla_\mathbb{B}u|^{p-2}tr\Bigg(\bigg(I+(p-2)\frac{\nabla_\mathbb{B}u^T}{|\nabla_\mathbb{B}u|}\frac{\nabla_\mathbb{B}u}{|\nabla_\mathbb{B}u|}\bigg)\nabla_\mathbb{B}^2u\Bigg).
     \end{align*}
We denote $$Q=I+(p-2)\frac{\nabla_\mathbb{B}u^T}{|\nabla_\mathbb{B}u|}\frac{\nabla_\mathbb{B}u}{|\nabla_\mathbb{B}u|},$$ then \eqref{eq:11}  can be reduced to \eqref{H3},
\begin{equation}\label{H3}
    t^{-p}|\nabla_\mathbb{B}u|^{p-2}tr(Q\nabla_\mathbb{B}^2u)+t^{-p}(n-p)|\nabla_\mathbb{B}u|^{p-2}(t\partial_t u)=f(t,x).
\end{equation}

\begin{Remark}\label{R1}
   Because $\frac{\nabla_\mathbb{B}u}{|\nabla_\mathbb{B}u|}$ is identity vector, the  eigenvalues of $\frac{\nabla_\mathbb{B}u^T}{|\nabla_\mathbb{B}u|}$$\frac{\nabla_\mathbb{B}u}{|\nabla_\mathbb{B}u|}$ are $1,0,\cdots$, $0$. So the eigenvalues of matrix $Q$ are $p-1,1,\cdots,1$.  Furthermore $ p\ge 2$, so 
    we can take $\lambda=1,\ \Lambda=p-1$ such that
    $$\lambda I \leq Q \leq \Lambda I.$$
 According to the Definition $\ref{D1}$, we can derive
 $$\mathcal{M}^-_{1,p-1}(\nabla_B^2u) \le tr(Q\nabla_B^2u) \le \mathcal{M}^+_{1,p-1}(\nabla_B^2u).$$
 For convenience, We denote $\mathcal{M}^-_{1,p-1}$ and $\mathcal{M}^+_{1,p-1}$  as $\mathcal{M}^-$ and $\mathcal{M}^+$, respectively.
 \end{Remark}
 \begin{Remark}\label{Remark2.3}
    In the following we discuss the necessary conditions for the viscosity solutions.\\
    $(1)$ If $u$ is a supersolution of equation \eqref{H3}, then  for all $\phi\in C^2(\mathbb{B}) $  the following inequality holds  at each local minimum point
    $(t_0,x_0)\in \mathbb{B} \ \text{of}\  u-\phi ,$ 
  $$
       t^{-p}|\nabla_\mathbb{B}\phi|^{p-2}tr(Q\nabla_\mathbb{B}^2\phi)+t^{-p}(n-p)|\nabla_\mathbb{B}\phi|^{p-2}(t\partial_t \phi)\le f(t_0,x_0). $$
       Then
       \begin{equation}  \label{H5} t^{-p}|\nabla_\mathbb{B}\phi|^{p-2}\mathcal{M}^-(\nabla_\mathbb{B}^2\phi)+t^{-p}(n-p)|\nabla_\mathbb{B}\phi|^{p-2}(t\partial_t \phi)\le f(t_0,x_0).\end{equation}
    $(2)$ If $u$ is a  subsolution of equation \eqref{H3}, then  for all $\phi\in C^2(\mathbb{B}) $  the following inequality holds  at each local maximum point
    $(t_0,x_0)\in \mathbb{B} \ \text{of}\  u-\phi ,$
    \begin{equation}
        t^{-p}|\nabla_\mathbb{B}\phi|^{p-2}\mathcal{M}^+(\nabla_\mathbb{B}^2\phi)+t^{-p}(n-p)|\nabla_\mathbb{B}\phi|^{p-2}(t\partial_t \phi)\ge f(t_0,x_0).
    \end{equation}
\end{Remark}

\section{\textbf{ABP and H\"older estimate}}
In this section, motivated by the work in \cite{IF}, \cite{IIV}, \cite{CI}, ect, we explore the properties of viscosity solutions to Dirichlet problem \eqref{H8}. Firstly, we  give the proof of Theorem \ref{T3} and  Theorem \ref{C} to obtain the ABP estimate of the viscosity solutions for \eqref{eq:11} and H\"older continuity of the viscosity solutions for  \eqref{H8}. Further we give  the proof of Theorem \ref{B} and get the existence of solutions of \eqref{H8}.

Compared with the work in articles mentioned above, our innovation mainly lies in the following three aspects.  Firstly, our model \eqref{eq:11} is new, mainly reflected in the degeneracy of the gradient operator itself in the model. Correspondingly, we  provide an appropriate distance in Definition \ref{cone 9}. Secondly,  the right hand control term of the ABP estimate \eqref{H11} is $L^{\infty}$ norm, different from \cite{IIV}, \cite{IFA}, ect. On the basis of \cite{CL30} research, we obtain the weak Hanack inequality \eqref{eq:44} through classical scaling transformation, which is a crucial step in the proof of ABP estimate.  Combined with the properties of the region, we ultimately obtained the global ABP estimate. Finally, We need to emphasize that  we have obtained a  weighted H\"older continuity \eqref{eq:55} rather than classical H\"older continuity. We combine the local ABP estimate and the properties of domain to obtain globally weighted H\"older continuity. Furthermore, we obtain the existence of viscosity solutions to \eqref{H8}  through function convergence and domain approximation.
%

     If $u$ is a supersolution of \eqref{H3}, then $u$ satisfies \eqref{H5} in viscosity sense. Let us define
$$u_m=\begin {cases}\min \{u,m\}  &\quad \ (t,x)\in A,\\m  &\quad \ (t,x) \in \Omega_d \backslash A,\end {cases}$$
where $A=\mathbb B \cap \Omega_d \neq\emptyset$, and $m=\inf _{(t,x)\in\{\partial A \cap \Omega_{d}\}}u(t,x)$, then $u_m$ satisfies \eqref{4.2} in $ \Omega_d$
\begin{equation}\label{4.2}
    t^{-p}|\nabla_\mathbb{B}u_m|^{p-2}\mathcal{M}^-(\nabla_\mathbb{B}^2u_m)+t^{-p}(n-p)|\nabla_\mathbb{B}u_m|^{p-2}(t\partial_t u_m)\le f^+(t,x) \chi(A),
\end{equation}
that is, $u_m $ is a viscosity supersolution of $\eqref{4.22}$
\begin{equation}\label{4.22}
    t^{-p}|\nabla_\mathbb{B}u_m|^{p-2}\mathcal{M}^-(\nabla_\mathbb{B}^2u_m)+t^{-p}(n-p)|\nabla_\mathbb{B}u_m|^{p-2}(t\partial_t u_m)= f^+(t,x) \chi(A),
\end{equation}
It is just to illustrate that  $(t,x)\in \partial A$. Since $u(t,x)\geq m$ for $(t,x)\in  \partial A$, we conclude that \ref{4.2} holds.


Firstly, we present two existing results from \cite{CL30} as the preparation.
\begin{lemma}[
\cite{CL30}]\label{Theorem3.2}
For any non-negative continuous function $u:\ \overline{B}_1 \rightarrow \mathbb{R}$ such that
$$
\begin{array}{ll}
\mathcal{P}^{-}\left(D^2 u, \nabla u\right) \leq C_0 & \text { in } B_1, \\
\mathcal{P}^{+}\left(D^2 u, \nabla u\right) \geq-C_0 & \text { in } B_1,
\end{array}
$$
where
$$
 \mathcal{P}^{-}\left(D^2 u, \nabla u\right)= \begin{cases}\lambda tr D^2 u^{+}-\Lambda tr D^2 u^{-}-\Lambda|\nabla u| & \text{ if }|\nabla u| \geq \gamma, \\
-\infty & \text{ otherwise, }\end{cases}
$$
 $$
\mathcal{P}^{+}\left(D^2 u, \nabla u\right)= \begin{cases}\Lambda tr D^2 u^{+}-\lambda tr D^2 u^{-}+\Lambda|\nabla u| & \text{ if }|\nabla u| \geq \gamma, \\
+\infty & \text{otherwise,}\end{cases} \\$$
we have
$$
\sup _{B_{1 / 2}} u \leq C\left(\inf _{B_{1 / 2}} u+C_0\right).
$$
The constant $C$ depends on $\lambda, \Lambda$, the dimension and $\gamma /\left(C_0+\mathop{\inf}\limits_{B_{1 / 2}} u\right)$.
\end{lemma}
 \begin{lemma}[
 \cite{CL30}]\label{Theome 3.1}
 There exist small constants $\epsilon_0>0$ and $\epsilon>0$ such that if $\gamma \leq \epsilon_0$, then for any lower semicontinuous function $u: B_2 \rightarrow \mathbb{R}$ such that
$$
\begin{aligned}
& u \geq 0 \quad\quad\quad\quad\quad\quad\ \  \text{ in } B_2, \\
& \mathcal{P}^{-}\left(D^2 u, \nabla u\right) \leq 1 \quad \text{ in } B_2, \\
& \inf _{B_1} u \leq 1,
\end{aligned}
$$
we have
$$
\left|\{u>t\} \cap B_1\right| \leq C t^{-\epsilon} \quad \text { for all } t>0 \text {. }
$$
\end{lemma}
We can easily see through their proof process that the result still holds when the coefficient of the first-order term is $(K_0d_0+1)|n-p|$, and we can choose $\epsilon_0=1$.

\begin{corollary}[Harnack  inequality]\label{corolly3}

For any non-negative continuous function $u: \ \overline{B}_1 \rightarrow \mathbb{R}$ such that
$$
\begin{array}{ll}
\mathcal{P}^{-}_{1}\left(D^2 u, \nabla u\right) \leq C_0 & \text{ in } B_1, \\
\mathcal{P}^{+}_{1}\left(D^2 u, \nabla u\right) \geq-C_0 & \text{ in } B_1,
\end{array}
$$
where
$$\begin{aligned}
& \mathcal{P}^{-}_{1}\left(D^2 u, \nabla u\right)\\=& \begin{cases}tr D^2 u^{+}-(p-1)tr D^2 u^{-}-(K_0d_0+1)|n-p||\nabla u| & \text{ if }|\nabla u| \geq \gamma, \\
-\infty & \text{otherwise, }\end{cases}\end{aligned}
$$
 $$\begin{aligned}
&\mathcal{P}^{+}_{1}\left(D^2 u, \nabla u\right)\\=& \begin{cases}(p-1) tr D^2 u^{+}- tr D^2 u^{-}+(K_0d_0+1)|n-p||\nabla u| & \text{ if }|\nabla u| \geq \gamma, \\
+\infty & \text{otherwise,}\end{cases} \end{aligned}$$
we have
$$
\sup _{B_{1 / 2}} u \leq C\left(\inf _{B_{1 / 2}} u+C_0\right).
$$
The constant $C$ depends on $p , n, K_0d_0 $ and $\gamma /\left(C_0+\mathop{\inf}\limits_{B_{1 / 2}} u\right)$.
\end{corollary}
\begin{corollary}\label{corollary4}

 There exists a small constant $\epsilon>0$ such that if $\gamma \leq 1$, then for any lower semicontinuous function $u: B_2 \rightarrow \mathbb{R}$ such that
$$
\begin{aligned}
& u \geq 0 \quad\quad\quad\quad\quad\quad\ \  \text{ in } B_2, \\
& \mathcal{P}^{-}_{1}\left(D^2 u, \nabla u\right) \leq 1 \quad \text{ in } B_2, \\
& \inf _{B_1} u \leq 1,
\end{aligned}
$$
we have
$$
\left|\{u>t\} \cap B_1\right| \leq C t^{-\epsilon} \quad \text{ for all } t>0 \text {,}
$$
with $\epsilon$ depending on $n,\ p,\ K_0d_0$.

\end{corollary}
\begin{lemma}\label{lemma3.31}
 If $u_m$ is a nonnegative supersolution of \eqref{4.22} in $\Omega_{2d}$ with $d\leq K_0d_0+1$, $A=\Omega_{2d}\cap \mathbb{B}$, then
   \begin{equation}\label{eq:44}
   \left(\frac{1}{\vert \Omega_{d}\vert }\int_{\Omega_{d}}u_m^{p_0}\frac{dt}{t}dx\right)^{\frac{1}{p_0}}\leq C\left(\mathop{\inf }\limits_{\Omega_{d}} u_m+d^{\frac{p}{p-1}} {\vert\vert t^pf^-\vert\vert^{\frac{1}{p-1}}_{L^{\infty}(\Omega_{2d}\cap A)}}\right),\end{equation}
    with $p_0$ and $C$ depending on $n,\ p,\ K_0d_0.$
\end{lemma}

\begin{proof}
Let's assume that $\Omega_{2d}$ is centered at $(t_0,x_0)\in \mathbb{B},$ 
  next, we will do a stretching transformation.

   Let 
 $$u_l(s,y)=\frac{1}{L}u_m(T(s,y)), \ T(s,y)=\left(s^d t_0^{1-d},x_0+d(y-x_0)\right),$$ then we have
 $$  \nabla_\mathbb{B}u_l=\frac{d}{L}\nabla_\mathbb{B}u_m, \  \nabla^2_\mathbb{B}u_l= \frac{d^2}{L}\nabla^2_\mathbb{B}u_m.$$
By calculation, it can be concluded that $u_l$ is the supersolution of the following function,
\begin{equation}
\begin{aligned}
   & {\vert \nabla_\mathbb{B}u_l\vert}^{p-2}  tr\left(\left(I+(p-1)\frac{(\nabla_\mathbb{B}u_l)^{T}(\nabla_\mathbb{B}u_l)}{{\vert \nabla_\mathbb{B}u_l \vert}^2}\right)\nabla_\mathbb{B}^2u_l\right)+d(n-p)\vert \nabla_\mathbb{B}u_l\vert^{p-2}\frac{t\partial u_l}{\partial t}\\
   & =\frac { \left(d^pf^+(t,x)t^{p} \chi(A)\right)\circ T}{L^{p-1}}.
  \end{aligned}
\end{equation}
We set
$$a=\ln s,\ \ \bar u_l(a,y)=u_l(s,y),$$ $$\left(\overline{\left(d^pf^+(t,x)t^{p} \chi(A)\right)\circ T}\right)(a,y)=\left({\left(d^pf^+(t,x)t^{p} \chi(A)\right)\circ T}\right)(s,y).$$
then
$$\nabla_\mathbb{B}^2u_l=\nabla^2 \bar u_l\ \ \text{and}\ \ \nabla_\mathbb{B} u_l=\nabla \bar u_l.$$
We obtain $\bar u_l(a,x)$ is the supersolution of the following function,
\begin{equation}
\mathcal{M}^-(\nabla^2\bar {u}_l(a,y))+d(n-p)(\partial_a \bar {u}_l(a,y))=\frac { \overline{\left(d^pf^+(t,x)t^{p} \chi(A)\right)\circ T}}{L^{p-1}|\nabla\bar{u}_l(a,y)|^{p-2}}.
\end{equation}
 If we choose
$$L=\mathop{\inf u_m}\limits_{\Omega_d}+\delta+ d^{\frac{p}{p-1}}{\vert\vert t^pf^+\vert\vert^{\frac{1}{p-1}}_{L^{\infty}({\Omega_{2d}})}},$$
$$|\nabla\bar{u}_l(a,y)|\geq 1, $$
we obtain that $\bar{u}_l$ satifies \\
$$ \inf _{B_1}\bar{u}_l \leq 1,$$
$$ \vert\vert \frac { \overline{\left(d^pf^+(t,x)t^{p} \chi(A)\right)\circ T}}{L^{p-1}|\nabla\bar{u}_l(a,y)|^{p-2}}\vert\vert_{L^{\infty}(B_2)} \leq 1.$$
Since $d\leq K_0d_0+1,$ then
 $$\mathcal{P}^{-}_{1}\left(D^2 \bar{u}_l, \nabla \bar{u}_l\right)\leq 1 \ \ \ \text { in } B_2$$ with $\gamma=1$.
  By Corollary \ref{corollary4}, we have
  $$
\left|\{\bar{u}_l>\tau\} \cap B_1\right| \leq C \tau^{-\epsilon} \quad \text { for all } \tau>0 \text {,}
$$
that is $$ \left|\{{u}_l>\tau\} \cap \Omega_1\right| \leq C \tau^{-\epsilon} \quad \text { for all } \tau>0 \text {.} $$
Then
$$\int \limits_{\Omega_1}u_l^{p_0}\frac{ds}{s}dy=p_0\int \limits_0^{+\infty}\tau^{p_0-1}|\left\{u_l\geq \tau\right\}\cap \Omega_1| d\tau,$$
 we  choose $p_0=\epsilon /2$ 
 in order to get
$$\frac{1}{p_0}\int \limits_{\Omega_1}u_l^{p_0}\frac{ds}{s}dy\le \int \limits_0^1\tau^{\epsilon/2-1}|\Omega_1|d\tau+ C\int \limits_1^{+\infty}\tau^{\epsilon/2-1}s^{-\epsilon}d\tau=:C,$$
$$  $$
Finally, we get
 $$\left(\frac{1}{\vert \Omega_{d}\vert }\int_{\Omega_{d}}u_m^{p_0}\frac{dt}{t}dx\right)^{\frac{1}{p_0}}\leq C\left(\mathop{\inf }\limits_{\Omega_{d}}u_m+d^{\frac{p}{p-1}} {\vert\vert t^pf^+\vert\vert^{\frac{1}{p-1}}_{L^{\infty}(\Omega_{2d}\cap A)}}\right),$$
where $p_0$ \text{and} $ C$ depend on $n,\ p,\ K_0d_0.$
\end{proof}
\begin{lemma}\label{lemma3.32}
 If $v\in C(\overline{\Omega}_{d})$ is a nonnegative viscosity solution of \eqref{H3} in $\Omega_{d}$ with $\Omega_{d}\subset\subset\mathbb{B}$ and $d\leq K_0d_0+1,$ then

    $$\sup _{\Omega_{d}} v\leq C(\inf _{\Omega_{d/2}} v+d^{\frac{p}{p-1}} {\vert\vert t^pf\vert\vert^{\frac{1}{p-1}}_{L^{\infty}(\Omega_{d})}}).$$
    with $C$ depending on $n,\ p,\ K_0d_0.$
\end{lemma}
\begin{proof}
Now, we set $a=lnt,\ \bar v(a,x) =v(t,x),\ \bar f(a,x)=f(t,x)$, we can get $\bar v$ is the viscosity solution of \eqref{HH3} 
\begin{equation}\label{HH3}
   |\nabla u|^{p-2}tr(Q\nabla ^2u)+(n-p)|\nabla u|^{p-2}(\partial _ a u)=\bar f(a,x)e^{ap}
\end{equation}
with $Q=I+(p-2)\frac{\nabla u^T}{|\nabla u|}\frac{\nabla u}{|\nabla u|}$.

We set $$\omega=\frac{\bar{v}(T(b, y))}{L}, \quad(a, x)=T(b, y)=d(b, y),$$
$$
L^{p-1}= \frac{||\bar{f}  e^{ap}||_{L^{\infty}(B_{d})}  d^p}{{(1+\mathop{\inf}\limits_{B_{1/2}} \omega)}^{p-2}},
$$
then $\omega$ is the viscosity solution of \eqref{HH4}
\begin{equation}\label{HH4}
\operatorname{tr}\left(Q \nabla^2 \omega\right)+d(n-p) \partial_a \omega= \frac{\left(\bar{f} e^{ap}\right) \circ(T(b, y)) d^p}{|\nabla \omega|^{p-2} L^{p-1}}.
\end{equation}
If $|\nabla \omega| \geqslant \gamma$,
then $\omega$ satisfies
$$\mathcal{M}^{-}\left( \nabla^2 \omega\right)+d(n-p) \partial_a \omega \leq \frac{\left(\bar{f}  e^{ap}\right)\circ(T(b, y)) d^p}{L^{p-1} {|\nabla \omega|}^{p-2}},$$
and $$\mathcal{M}^{+} \left( \nabla^2 \omega\right)+d(n-p) \partial_a \omega \geq \frac{\left(\bar{f}  e^{ap}\right)\circ(T(b, y)) d^p}{L^{p-1} {|\nabla \omega|}^{p-2}}.$$
Then $\omega$ satisfies the assumptions of Corollary \ref{corolly3} with $C_0=1$ and $\gamma=1+\mathop{\inf}\limits _{B_{1 / 2}} \omega$, so
$$
\sup _{B_{1 / 2}} \omega \leq C\left(\inf _{B_{1 / 2}} \omega+1\right),
$$

$$
 \sup _{\Omega_{d/2}} v\leq C(\inf _{\Omega_{d/2}} v+d^{\frac{p}{p-1}} {\vert\vert t^pf\vert\vert^{\frac{1}{p-1}}_{L^{\infty}(\Omega_{d} )}}).
 $$
\end{proof}
 \begin{proof}[ \textbf{Proof of Theorem \ref{T3}}]
     Let $M=\mathop{\sup \limits_{\mathbb{B}} v^+}$, $ u=M-v^+\geq0 $, since $v$ is a subsolution viscosity of \eqref{eq:11}, then $u$ is the supersolution of
$$t^{-p}|\nabla_\mathbb{B}u|^{p-2}\mathcal{M}^-(\nabla_\mathbb{B}^2u)+t^{-p}(n-p)|\nabla_\mathbb{B}u|^{p-2}(t\partial_t u)=f^-(t,x),\ \ (t,x)\in\mathbb{B}.$$
For all fix $(t,x)\in \mathbb{B}$, let $ A=\mathbb{B}\cap\Omega_{2\tilde{R}_{(t,x)}}$,
therefore, $u_m$ is the supersolution of
$$t^{-p}|\nabla_\mathbb{B}u_m|^{p-2}\mathcal{M}^-(\nabla_\mathbb{B}^2u_m)+t^{-p}(n-p)|\nabla_\mathbb{B}u_m|^{p-2}(t\partial_t u_m)=\chi(A)f^-(t,x)$$ in $\Omega_{2\tilde{R}_{(t,x)}}$.
By Lemma \ref{lemma3.31}, we can get that $u_m$ satisfies
$$(\frac{1}{\vert \Omega_{\tilde{R}_{(t,x)}}\vert }\int_{\Omega_{\tilde{R}_{(t,x)}}}u_m^{p_0}\frac{dt}{t}dx)^{\frac{1}{p_0}}\leq C(\mathop{\inf u_m}\limits_{\Omega_{\tilde{R}_{(t,x)}}}+\tilde{R}^{\frac{p}{p-1}}_{(t,x)} {\vert\vert t^pf^-\vert\vert^{\frac{1}{p-1}}_{L^{\infty}(\Omega_{2\tilde{R}_{(t,x)}}\cap A)}}),$$
with $C$ depends on $ n,\ p,\ K_0d_0$. 

Since $\mathop{\inf }\limits_{\Omega_{\tilde{R}_{(t,x)}}}u_m\leq \mathop{\inf }\limits_{A\cap \Omega_{\tilde{R}_{(t,x)}}}u$, then
$$(\frac{1}{\vert \Omega_{\tilde{R}_{(t,x)}} \vert }\int_{\Omega_{\tilde{R}_{(t,x)}} }u_m^{p_0}\frac{dt}{t}dx)^{\frac{1}{p_0}}\leq C ( \mathop{\inf u}\limits_{A\cap \Omega_{\tilde{R}_{(t,x)}}} +\tilde{R}^{\frac{p}{p-1}}_{(t,x)}{\vert\vert t^pf^-\vert\vert^{\frac{1}{p-1}}_{L^{\infty}(\Omega_{2\tilde{R}_{(t,x)}}\cap \mathbb{B})}}), $$

$$
\begin{aligned}
(\frac{1}{|\Omega_{\tilde{R}_{(t,x)}}|}\int_{ \Omega_{\tilde{R}_{(t,x)}}  }u_m^{p_0}\frac{dt}{t}dx)^{\frac{1}{p_0}}&\geq (\frac{1}{\Omega_{\tilde{R}_{(t,x)}}}\int_{\Omega_{\tilde{R}_{(t,x)}} \backslash \mathbb{B} }u_m^{p_0}\frac{dt}{t}dx)^{\frac{1}{p_0}}\\&=m (\frac{\vert \Omega_{\tilde{R}_{(t,x)}} \backslash \mathbb{B} \vert}{\vert \Omega_{\tilde{R}_{(t,x)}} \vert})^{\frac{1}{p_0}}\geq m \sigma ^{\frac{1}{p_0}}.
\end{aligned}$$
Since
$$\mathop{\inf u}\limits_{A \cap \Omega_{\tilde{R}_{(t,x)}}}\leq u(t,x)=M-v^+(t,x),$$
then
$$ m \sigma ^{\frac{1}{p_0}} \leq C(M-v^+(t,x)+\tilde{R}^{\frac{p}{p-1}}_{(t,x)} {\vert\vert t^pf^-\vert\vert^{\frac{1}{p-1}}_{L^{\infty}(\Omega_{2\tilde{R}_{(t,x)}}\cap \mathbb{B})}}).$$
And $ m=M-\mathop{\sup v^+ }\limits_{\partial{\mathbb{B}_{2\tilde{R}_{(t,x)}}}\cap \Omega_{2\tilde{R}_{(t,x)}}} \geq M -\mathop{\sup }\limits_{\partial\mathbb{B}}v^+$, so
\begin{equation}\label{equation4.16}
v^+(t,x)\leq (1-\frac{\sigma ^{\frac{1}{p_0}}}{C})M+\frac{\sigma ^{\frac{1}{p_0}}}{C}\mathop{\sup }\limits_{\partial\mathbb{B}}v^++\tilde{R}^{\frac{p}{p-1}}_{(t,x)} {\vert\vert t^pf^-\vert\vert^{\frac{1}{p-1}}_{L^{\infty}(\Omega_{2\tilde{R}_{(t,x)}}\cap \mathbb{B})}},
\end{equation}
$$
\mathop{\sup }\limits_{\mathbb{B}}v^+ \leq\mathop{\sup }\limits_{\partial\mathbb{B}}v^++C\mathop {\sup }\limits_{\mathbb{B}}\tilde{R}^{\frac{p}{p-1}}_{(t,x)} {\vert\vert t^pf^-\vert\vert^{\frac{1}{p-1}}_{L^{\infty}(\Omega_{2\tilde{R}_{(t,x)}}\cap \mathbb{B})}},
$$
$$
\mathop{\sup }\limits_{\mathbb{B}}v^+ \leq\mathop{\sup}\limits_{\partial\mathbb{B}} v^++C(K_0d_0)^{\frac{p}{p-1}}\mathop {\sup }\limits_{\mathbb{B}} {\vert\vert t^pf^-\vert\vert^{\frac{1}{p-1}}_{L^{\infty}(\Omega_{2\tilde{R}_{(t,x)}}\cap \mathbb{B})}},
$$
where $C$ depends on $n,\ p,\ \sigma,\ K_0d_0$.
 \end{proof}

 \begin{proof}[ \textbf{Proof of Corollary \ref{Theorem4.2}}]
    If $v$ is the supersolution viscosity of \eqref{eq:11}, then $-v$ is the subsolution of \eqref{eq:11} with $-f$ instead of $f$. By Theorem \ref{T3} , we can conclude that
$$
\mathop{\sup }\limits_{\mathbb{B}}(-v)^+\leq \mathop{\sup }\limits_{\partial \mathbb {B}}(-v)^++C(K_0d_0)^{\frac{p}{p-1}}\mathop{\sup}\limits_{\mathbb{B}}{\vert\vert t^pf^+\vert\vert^{\frac{1}{p-1}}_{L^{\infty}(\Omega_{2\tilde{R}_{(t,x)}}\cap \mathbb{B})}},
$$
$$
\mathop{\sup }\limits_{\mathbb{B}}v^-\leq \mathop{ \sup }\limits_{\partial \mathbb{B}}v^-+C(K_0d_0)^{\frac{p}{p-1}}\mathop{\sup}\limits_{ \mathbb{B}}{\vert\vert t^pf^+\vert\vert^{\frac{1}{p-1}}_{L^{\infty}(\Omega_{2\tilde{R}_{(t,x)}}\cap \mathbb{B})}},
$$
$$
\mathop{\sup }\limits_{\mathbb{B}}\vert v \vert\leq \mathop{\sup }\limits_{\partial \mathbb{B}}\vert v \vert+C(K_0d_0)^{\frac{p}{p-1}}\mathop{\sup}\limits_{\mathbb{B}} {\vert\vert t^pf\vert\vert^{\frac{1}{p-1}}_{L^{\infty}(\Omega_{2\tilde{R}_{(t,x)}}\cap \mathbb{B})}}.
$$
\end{proof}
\begin{proof}[ \textbf{Proof of Theorem \ref{C}}]
    Let $\Omega_{\tilde{r}}$ be a ball of radius $\tilde{r} $ centered at $ (t,x)$. We set
    \begin{equation}
        M_{\tilde{r}}=\mathop{\sup}\limits_{\mathbb B\cap \Omega_{\tilde{r}} }v,\ m_{\tilde{r}}=\mathop{\inf}\limits_{\mathbb B\cap \Omega_{\tilde{r}} }v,\ \omega_{\tilde{r}}=M_{\tilde{r}}-m_{\tilde{r}}.
    \end{equation}

 Firstly, for fixed $(t,x)$,  we can change the $ M$ in the Theorem \ref{T3} to be $\mathop{\sup }\limits_{\mathbb{B}_{2\tilde{R}_{(t,x)}}}v^+$, similar to \eqref{equation4.16}, we find that 
  $$
\frac{v^+(t,x)}{d_{(t,x)}^{\alpha}}\leq (1-\frac{1}{C}\sigma ^{\frac{1}{p_0}})\frac{\mathop{\sup }\limits_{\mathbb{B}_{2\tilde{R}_{(t,x)}}}v^+}{d_{(t,x)}^{\alpha}}+\frac{\tilde{R}^{\frac{p}{p-1}}_{(t,x)} {\vert\vert t^pf^-\vert\vert^{\frac{1}{p-1}}_{L^{\infty}(\Omega_{2\tilde{R}_{(t,x)}}\cap \mathbb{B})}}}{d_{(t,x)}^{\alpha}}
.$$
Since $v=0 \ \text{on} \ \partial{\mathbb{B}}$, so we can find a $(t^*,x^*)$ such that
$$\frac{v^+(t,x)}{d_{(t,x)}^{\alpha}}\leq (1-\frac{1}{C}\sigma ^{\frac{1}{p_0}})\frac{v^+(t^*,x^*)}{d_{(t,x)}^{\alpha}}+\frac{\tilde{R}^{\frac{p}{p-1}}_{(t,x)} {\vert\vert t^pf^-\vert\vert^{\frac{1}{p-1}}_{L^{\infty}(\Omega_{2\tilde{R}_{(t,x)}}\cap \mathbb{B})}}}{d_{(t,x)}^{\alpha}}
.$$
Since $d_{(t^*,x^*)}\leq 3\tilde{R}_{(t,x)}$, and $\tilde{R}_{(t,x)}\leq K_0d_{(t,x)}$, so
$$ \frac{v^+(t,x)}{d_{(t,x)}^{\alpha}}\leq (1-\frac{1}{C}\sigma ^{\frac{1}{p_0}})(3{K_0})^{\alpha}\frac{v^+(t^*,y^*)}{d_{(t^*,x^*)}^{\alpha}}+\frac{\tilde{R}^{\frac{p}{p-1}}_{(t,x)} {\vert\vert t^pf^-\vert\vert^{\frac{1}{p-1}}_{L^{\infty}(\Omega_{2\tilde{R}_{(t,x)}}\cap \mathbb{B})}}}{d_{(t,x)}^{\alpha}}
.$$
 Then for all $\alpha$ such that $(1-\frac{1}{C}\sigma ^{\frac{1}{p_0}})(3{K_0})^{\alpha}< 1$, we have
    $$
\mathop{\sup}\limits_{\mathbb{B}}\frac{v^+(t,x)}{d_{(t,x)}^{\alpha}}\leq C'K_0^{\frac{p}{p-1}}d_0^{\frac{p}{p-1}-\alpha}\mathop{\sup}\limits_{\mathbb{B}} {\vert\vert t^pf^-\vert\vert^{\frac{1}{p-1}}_{L^{\infty}(\Omega_{2\tilde{R}_{(t,x)}}\cap \mathbb{B})}},$$
where $C'$ depends on $n,\ p,\ \sigma,\ K_0d_0,\ \alpha$.
    Similarly, like Corollary \ref{Theorem4.2}, we perform a similar operation on $-v$ and obtain
$$
\mathop{\sup}\limits_{\mathbb{B}}\frac{v^-(t,x)}{d_{(t,x)}^{\alpha}}\leq C'K_0^{\frac{p}{p-1}}d_0^{\frac{p}{p-1}-\alpha}\mathop{\sup}\limits_{\mathbb{B}} {\vert\vert t^pf^+\vert\vert^{\frac{1}{p-1}}_{L^{\infty}(\Omega_{2\tilde{R}_{(t,x)}}\cap \mathbb{B})}},
$$
that is
\begin{equation}\label{equation4.19}
\mathop{\sup}\limits_{\mathbb{B}}\frac{|v(t,x)|}{d_{(t,x)}^{\alpha}}\leq C' K_0^{\frac{p}{p-1}}d_0^{\frac{p}{p-1}-\alpha}\mathop{\sup}\limits_{ \mathbb{B}} {\vert\vert t^pf\vert\vert^{\frac{1}{p-1}}_{L^{\infty}(\Omega_{2\tilde{R}_{(t,x)}}\cap \mathbb{B})}}.
\end{equation}
Next, we will discuss in three different situations.

\noindent$(1)$  If $\tilde{r}< 1$, and $\tilde{r}>\frac{1}{2}d_{(t,x)}$, then for $(s,y)\in \Omega_{\tilde{r}}$, we have
    $$ d_{(s,y)}\leq d_{(t,x)}+d\left( (t,x),(s,y)\right)\leq 3\tilde{r},$$ then
    \begin{equation}\label{3.136}
    \begin{aligned}
        \omega(\tilde{r})&=\mathop{\sup}\limits_{\mathbb B\cap \Omega_{\tilde{r}} }v-\mathop{\inf}\limits_{\mathbb B\cap \Omega_{\tilde{r}} }v\leq 2\mathop{\sup}\limits_{\mathbb B\cap \Omega_{\tilde{r} }}|v|\leq 2\mathop{\sup}\limits_{\mathbb B\cap \Omega_{\tilde{r} } }\left(\frac{3\tilde{r} }{d_{(s,y)}}\right)^{\alpha}|v(s,y)|\\&\leq
       2\cdot 3^{\alpha}\tilde{r}^{\alpha}\mathop{\sup}\limits_{\mathbb B\cap \Omega_{\tilde{r} } }\frac{|v(s,y)| }{d_{(s,y)}^{\alpha}}
       \\&\leq 2\cdot 3^{\alpha} C'K_0^{\frac{p}{p-1}}d_0^{\frac{p}{p-1}-\alpha}\tilde{r}^{\alpha}\mathop{\sup}\limits_{\mathbb{B}} {\vert\vert t^pf\vert\vert^{\frac{1}{p-1}}_{L^{\infty}(\Omega_{2\tilde{R}_{(t,x)}}\cap \mathbb{B})}}.
        \end{aligned}
    \end{equation}

\noindent $(2)$ If $\tilde{r}< 1$, and $\tilde{r}\leq \frac{1}{2}d_{(t,x)} $,
then $\Omega_{\tilde{r}/ 2} \subset \Omega_{\tilde{r}}\subset\subset \mathbb{B}$. 
 Since $\Omega_{\tilde{r}}\subset\Omega_{2\tilde{R}_{(t,x)}}$, we apply  $M_{\tilde{r}}-v$ and $v-m_{\tilde{r}}$ to the Lemma \ref{lemma3.32} and  get
%

$$
\begin{aligned}
& M_{\tilde{r}}-m_{{\tilde{r}} / 2} \leq C\left(M_{\tilde{r}}-M_{{\tilde{r}} / 2}+\tilde{r}^{\frac{p}{p-1}} {\vert\vert t^pf\vert\vert^{\frac{1}{p-1}}_{L^{\infty}(\Omega_{2\tilde{R}_{(t,x)}} ) }}\right), \\
& M_{{\tilde{r}} / 2}-m_{\tilde{r}} \leq C\left(m_{{\tilde{r}} / 2}-m_{\tilde{r}}+\tilde{r}^{\frac{p}{p-1}}  {\vert\vert t^pf\vert\vert^{\frac{1}{p-1}}_{L^{\infty}(\Omega_{2\tilde{R}_{(t,x)}} )}} \right),
\end{aligned}
$$
then we have
\begin{equation}\label{3.20}
\begin{aligned}
\omega({\tilde{r}}/ 2) & \leq \frac{C-1}{C+1} \omega({\tilde{r}})+\frac{2 C}{C+1} \tilde{r}^{\frac{p}{p-1}} {\vert\vert t^pf\vert\vert^{\frac{1}{p-1}}_{L^{\infty}(\Omega_{2\tilde{R}_{(t,x)}} )}} \\
&  \leq \frac{C-1}{C+1} \omega({\tilde{r}})+\frac{2 C}{C+1} {\tilde{r}}^{\alpha}\mathop{\sup}\limits_{ \mathbb{B}} {\vert\vert t^pf\vert\vert^{\frac{1}{p-1}}_{L^{\infty}(\Omega_{2\tilde{R}_{(t,x)}}\cap \mathbb{B})}}.
\end{aligned}
\end{equation}
From \eqref{3.136} and \eqref{3.20}, we obtain
\begin{equation}\label{3.17}
\omega({\tilde{r}}/ 2)  \leq C_1 \omega({\tilde{r}})+C_2 {\tilde{r}}^{\alpha}\mathop{\sup}\limits_{ \mathbb{B}} {\vert\vert t^pf\vert\vert^{\frac{1}{p-1}}_{L^{\infty}(\Omega_{2\tilde{R}_{(t,x)}}\cap \mathbb{B})}}, \end{equation}
with $C_1=\frac{C-1}{C+1} $ and $C_2=\frac{1}{C+1}\max\{{2 C},4\cdot 3^{\alpha} C'K_0^{\frac{p}{p-1}}d_0^{\frac{p}{p-1}-\alpha}\}$.

By the Lemma 8.23 of \cite{DN}, and \eqref{3.17} we get
\begin{equation}\label{3.23}
\begin{aligned}
\omega(\tilde{r}) &\leq C_3\left(\omega(1) \tilde{r}^\beta+ C_2 {\tilde{r}^{\alpha\mu}}\mathop{\sup}\limits_{ \mathbb{B}} {\vert\vert t^pf\vert\vert^{\frac{1}{p-1}}_{L^{\infty}(\Omega_{2\tilde{R}_{(t,x)}}\cap \mathbb{B})}}\right)\\&\leq
C_3\mathop{\sup}\limits_{ \mathbb{B}} {\vert\vert t^pf\vert\vert^{\frac{1}{p-1}}_{L^{\infty}(\Omega_{2\tilde{R}_{(t,x)}}\cap \mathbb{B})}}\left( 2C'K_0^{\frac{p}{p-1}}d_0^{\frac{p}{p-1}}\tilde{r}^\beta+ C_2 {\tilde{r}^{\alpha\mu}}\right)\\&
\leq C_4\mathop{\sup}\limits_{ \mathbb{B}} {\vert\vert t^pf\vert\vert^{\frac{1}{p-1}}_{L^{\infty}(\Omega_{2\tilde{R}_{(t,x)}}\cap \mathbb{B})}}\tilde{r}^{\rho}
\end{aligned}
\end{equation}
for any $\mu \in(0,1)$, $\rho\leq \min\{\alpha\mu,\beta\}$, where $C_3=C_3(C_1)$, $\beta=\beta(C_1, \mu)$ and $C_4=C_3\cdot \max\{2C'K_0^{\frac{p}{p-1}}d_0^{\frac{p}{p-1}},C_2\}$.

\noindent $(3)$ If $\tilde{r}\geq 1$, then
\begin{equation}\label{3.16}
    \begin{aligned}
        \omega(\tilde{r})&=\mathop{\sup}\limits_{\mathbb B\cap \Omega_{\tilde{r}} }v-\mathop{\inf}\limits_{\mathbb B\cap \Omega_{\tilde{r}} }v\leq 2\mathop{\sup}\limits_{\mathbb B\cap \Omega_{\tilde{r} }}|v|\\&\leq
       2\cdot {d_0}^{\alpha}\tilde{r}^{\rho}\mathop{\sup}\limits_{\mathbb B\cap \Omega_{\tilde{r} } }\frac{|v(s,y)| }{d_{(s,y)}^{\alpha}}
       \\&\leq 2\cdot C'K_0^{\frac{p}{p-1}}d_0^{\frac{p}{p-1}}\mathop{\sup}\limits_{\mathbb{B}}{\vert\vert t^pf\vert\vert^{\frac{1}{p-1}}_{L^{\infty}(\Omega_{2\tilde{R}_{(t,x)}}\cap \mathbb{B})}}.
        \end{aligned}
    \end{equation}
    In a word, we can get
    $$ ||v(t,x)||_{\rho,\mathbb B}\leq C\mathop{\sup}\limits_{\mathbb{B}} {\vert\vert t^pf\vert\vert^{\frac{1}{p-1}}_{L^{\infty}(\Omega_{2\tilde{R}_{(t,x)}}\cap \mathbb{B})}}.$$
    Next, when $(t,x)\in \mathbb{B}$ and $(s,y)\in \partial\mathbb{B}$,
   we can choose $\left(s^{\prime}, y^{\prime}\right) \in \mathbb{B}$ such that $d_{\left(s^{\prime}, y^{\prime}\right)} \leq 1$ and $\left(s^{\prime}, y^{\prime}\right)\in [(t,x),(s,y)]$, then
$$
\begin{aligned}
\frac{|v(t, x)-v(s, y)|}{|(t, x)-(s, y)|^\rho} & \leqslant \frac{\mid v(t, x)-v\left(s^{\prime}, y^{\prime} \mid\right.}{|(t, x)-(s, y)|^\rho}+\frac{\left|v\left(s^{\prime}, y^{\prime}\right)-0\right|}{|(t, x)-(s, y)|^\rho} \\
& \leq C \mathop{\sup}\limits_{\mathbb{B}} {\vert\vert t^pf\vert\vert^{\frac{1}{p-1}}_{L^{\infty}(\Omega_{2\tilde{R}_{(t,x)}}\cap \mathbb{B})}}+\left|\frac{v\left(s^{\prime}, y^{\prime}\right)}{d^\rho\left(s^{\prime}, y^{\prime}\right)}\right|.
\end{aligned}
$$
Based on \eqref{equation4.19}, $p \leq \alpha$ and
 $d\left(s^{\prime}, y^{\prime}\right)  \leq 1 $, we deduce that $$\begin{aligned}
\left|\frac{v\left(s^{\prime}, y^{\prime}\right)}{d^\rho\left(s^{\prime}, y^{\prime}\right)}\right|&\leq\left|\frac{v\left(s^{\prime}, y^{\prime}\right)}{d^\alpha\left(s^{\prime}, y^{\prime}\right)}\right|\\&
\leq C' K_0^{\frac{p}{p-1}}d_0^{\frac{p}{p-1}-\alpha}\mathop{\sup}\limits_{\mathbb{B}} {\vert\vert t^pf\vert\vert^{\frac{1}{p-1}}_{L^{\infty}(\Omega_{2\tilde{R}_{(t,x)}}\cap \mathbb{B})}}.
\end{aligned}
$$
Finally, we can get
    $$
    ||v(t,x)||_{\rho,\overline{\mathbb B}}\leq C\mathop{\sup}\limits_{\mathbb{B}} {\vert\vert t^pf\vert\vert^{\frac{1}{p-1}}_{L^{\infty}(\Omega_{2\tilde{R}_{(t,x)}}\cap \mathbb{B})}}.
    $$


\end{proof}

 \begin{proof}[\textbf{Proof of Theorem \ref{B}}]
Let $H_{j}$ be a sequence of smooth and bounded domains such that
$$H_{j}\subset\subset H_{j+1}\subset\subset \mathbb{B} \ \ \ \text{and } \ \ \ \cup_{j} H_{j} =\mathbb{B},$$
and
by \cite{IF22} Proposition 1, we can find a solution $u_j\in C(\overline{H}_j)$ for
\begin{equation}\label{H9}
    \begin{cases}
F((t,x),D_\mathbb{B}u,D_\mathbb{B}^2u)=0  \ \ \ & (t,x)\in  H_{j},\\
u=0  \ & (t,x)\in    \partial{H_{j}}.
\end{cases}
\end{equation}
Through Theorem \ref{C}, we can get
$$\begin{aligned}
||u_j(t,x)||_{\rho,\overline{H}_j}&\leq C\mathop{\sup}\limits_{ H_j} {\vert\vert t^pf\vert\vert^{\frac{1}{p-1}}_{L^{\infty}(\Omega_{2\tilde{R}_{(t,x)}}\cap H_j)}}\\&\leq C\mathop{\sup}\limits_{\mathbb{B}} {\vert\vert t^pf\vert\vert^{\frac{1}{p-1}}_{L^{\infty}(\Omega_{2\tilde{R}_{(t,x)}}\cap \mathbb{B})}}.\\
 \end{aligned}$$
 Therefore we can find a subsequence of $\{u_{j_k}\}$ converges uniformly to $v_{\bar{O}}$ in $\bar{O}$, where $O\subset \mathbb{B}$ is a bounded open set. By \cite{GPA} Proposition 4.1,  $v_{\bar{O}}$ is the viscosity solution of \eqref{equation3.12} in $ O$,
 \begin{equation}\label{equation3.12}
 F((t,x),D_\mathbb{B}u,D_\mathbb{B}^2u)=0 .\end{equation}
  Let $\bar{O}_{i}$ be a sequence compact sets such that
$$O_{i}\subset\subset O_{i+1}\subset\subset \mathbb{B} \ \ \ \text{and } \ \ \ \cup_{i} O_{i} =\mathbb{B},$$ and we define $$ v=\lim\limits_{i\to\infty}v_{O_i}.$$
For all $(t_0,x_0)\in \mathbb{B}$, we can choose $i$ enough large such that $(t_0,x_0)\in O_i$.
If $v-{\phi}$ attains a  local minimum at $(t_0,x_0)\in \mathbb{B}$ for all ${\phi}\in C^2(\mathbb{B})$, then $v_{O_i}-{\phi}$ attains  a  local minimum at $(t_0,x_0)\in \mathbb{B}$. So the following inequality holds $$F((t,x),D_\mathbb{B}{\phi},D_\mathbb{B}^2{\phi})\leq 0.
 $$
$v$ is the viscosity solution of \eqref{equation3.12}
  in $ \mathbb{B}$.

   For all $(t_0,x_0)\in \mathbb{B}$, there exists $ O_i$ such that $(t_0,x_0)\in O_i$. Since $\{u_{j_k}\}$ converges uniformly to $v_{\bar{O}_i}$ in $\bar{O_i}$, we choose $j_k$ enough large such that $\bar{O}_i\in H_{j_k},$ and
 $$ |v(t_0,x_0)-u_{j_k}(t_0,x_0)|\leq \frac{\varepsilon}{3}. $$
 For all $(s,y)\in \partial\mathbb{B}$, when $$|(t_0,x_0),(s,y)|\leq {\frac{\varepsilon}{3C}}^{\frac{1}{\rho}} {\mathop{\sup}\limits_{ \mathbb{B}} {\vert\vert t^pf\vert\vert^{\frac{1}{p-1}}_{L^{\infty}(\Omega_{2\tilde{R}_{(t,x)}}\cap \mathbb{B})}}}^{\frac{-1}{\rho}},$$ we have
 $$
 \begin{aligned}|v(t_0,x_0)-0|&\leq  |v(t_0,x_0)-u_{j_k}(t_0,x_0)|+ |u_{j_k}(t_0,x_0)-0|\\&\leq
 \frac{\varepsilon}{3}+\frac{\varepsilon}{3}<\varepsilon,
  \end{aligned}$$
 that is
 $$ \lim\limits_{(t_0,x_0)\to \partial\mathbb{B}}v(t_0,x_0)=0.$$
 We can define $v\big|_{\partial\mathbb{B}}=0$,
 then $v$ satisfies \eqref{H8}.
And we obtain \eqref{T1.10} by Theorem \ref{C}.
\end{proof}

\section{\textbf{Comparison principle of viscosity solution}}
In this section, we give the proof of Theorem \ref{A}  inspired by article \cite{MHP}, \cite{HP2},\ and \cite{HMS}, 
that is, the comparison principle of the viscosity solutions of the equation \eqref{eq:11}.

Compared to the research work on viscosity solutions in \cite{MHP} and \cite{HP2}, We mainly have the following two innovative points. Firstly, in response to the challenges caused by the absence of zero order terms, we transform the original equation \eqref{eq:11} into a new one \eqref{H9} by compounding a special monotonically increasing function. Through the comparison principle of viscosity solutions of the new equation  \eqref{H9} and the properties of the monotonically increasing function, we ultimately obtain the comparison principle of viscosity solutions to \eqref{eq:11}. Secondly, the degenerate cone gradient operator itself poses some difficulties in analysis. We construct the contradiction by selecting the appropriate maximum value point of the auxiliary function and obtain the  comparison principle of viscosity solutions through proof by contradiction.

Before showing the Theorem \ref{A}, let's first give some lemmas  and  propositions.
 \begin{lemma}\label{Proposition3.1}Let $G$ be a real-valued continuous function on $\Gamma=\mathbb B\times R \times R^n \times S^n$. Let $u$ be an $u.s.c.$ viscosity subsolution of
$$
G((t,x),u,\nabla_\mathbb{B}u,\nabla_\mathbb{B}^2u)=0 \ \ \ \ \ \ \ \ (t,x)\in \mathbb B.
$$
Let $\epsilon >0$, and define a continuous function $G^{\epsilon}$ on $\Gamma _\epsilon =\mathbb{B}_\epsilon \times R \times R^n \times S^n$ by
$$
G^\epsilon((t,x),r,p,X)=\max \{G((s,y),\tau,p,X)\ \big{|}\ |(e^s,y)-(e^t,x)|^2+(\tau-r)^2\leq \epsilon^2\},
$$
for $((t,x),r,p,X)\in \Gamma_\epsilon$. Then the upper $\epsilon$-envelope $u^\epsilon$ of $u$ is a viscosity subsolution of
$$
G^\epsilon((t,x),u,\nabla_\mathbb{B}u,\nabla_\mathbb{B}^2u)=0 \ \ \ \ \ \ \ \ (t,x)\in \mathbb B_\epsilon.
$$
\end{lemma}
\begin{proof}
Let $\phi \in C^2\left(\mathbb{B}_{\epsilon}\right)$ and $(t_0,x_0) \in \mathbb{B}_{\epsilon}$ satisfy
$$
\left(u^{\epsilon}-\phi\right)\left(t_0,x_0\right)=\text{local} \max \left(u^{\epsilon}-\phi\right) .
$$
By the definition of the upper $\epsilon$-envelope, we see that
$$
u^{\epsilon}\left(t_0,x_0\right)=u\left(s_0,y_0\right)+\left(\epsilon^2-\left|(e^{t_0},x_0)-(e^{s_0},y_0)\right|^2\right)^{1 / 2},$$
 for some $(s_0,y_0) $ satisfying  $|(e^{s_0},y_0)-(e^{t_0},x_0)|\leq \epsilon$.
And hence
$$
\begin{aligned}
&u(s,y)+\left(\epsilon^2-|(e^{t},x)-(e^{s},y)|^2\right)^{1 / 2}-\phi(t,x)\\ \leq &u\left(s_0,y_0\right)+\left(\epsilon^2-\left|(e^{t_0},x_0)-(e^{s_0},y_0)\right|^2\right)^{1 / 2}-\phi\left(t_0,x_0\right),
\end{aligned}
$$
for all $(t,x) \in U,$ a neighborhood of $(t_0,x_0)$, and $(s,y)$ satisfying $|(e^{s},y)-(e^{t},x)|\leq \epsilon$. We choose $y=x-x_0+y_0$ and $s=t-t_0+s_0$ in this last inequality to find that the function $$(s,y) \rightarrow u(s,y)-\phi\left(s-s_0+t_0,y-y_0+x_0\right)$$ attains a local  maximum at $(s_0,y_0)$. Therefore, by the definition of the viscosity subsolution, we have
$$
G\left((s_0,y_0), u\left(s_0,y_0\right), \nabla_{\mathbb{B}} \phi\left(t_0,x_0\right), \nabla_{\mathbb{B}}^2 \phi\left(t_0,x_0\right)\right) \geq 0 .
$$
Since $$\left|(e^{s_0},y_0)-(e^{t_0},x_0)\right|^2+\left(u\left(s_0,y_0\right)-u^{\epsilon}\left(t_0,x_0\right)\right)^2=\epsilon^2,$$ the proof is concluded.
\end{proof}
The proof of Proposition \eqref{Proposition3.2} refers to the idea of Proposition 4.3 in \cite{HMS}. Compared to classical gradient operators, the cone degenerate operator makes the analysis of Hessian matrix  more complex. We obtained \eqref{eq:22}by discussing in categories and applying properties of matrices such as rank, eigenvalues, trace, etc.
 \begin{Proposition}\label{Proposition3.2}
 Given $\epsilon>0$ and $\phi \in C^1(\mathbb{B})$. Let $u\in USC(\mathbb{B})$ and $u^{\epsilon}$ is the upper $\epsilon$-envelope of $u$. Assume $u^{\epsilon}-\phi$ attains its maximum at $(t_0,x_0) \in \mathbb{B}_{\epsilon}$, then there exist a constants  $C$, depending on $\epsilon$ and $\nabla\phi(t_0,x_0)$, and a neighborhood $U$ of $(t_0,x_0)$ such that for any $(t,x)\in U$ the following facts hold,
\begin{equation}\label{eq:22}
 \nabla_{\mathbb{B}}^2u^{\epsilon}(t,x) \geq
 CI,
\end{equation}
\begin{equation}
\nabla^2u^{\epsilon}(t,x) \geq
CI,
\end{equation}
with $C=-\epsilon^2\left(\epsilon^2-(\delta+2 \gamma)^2\right)^{-3 / 2}.$
\end{Proposition}
 \begin{proof}We choose $(s_0,y_0) $ satisfying $|(e^{s_0},y_0)-(e^{t_0},x_0)|< \epsilon$ such that
$$
u^\epsilon\left(t_0,x_0\right)=u\left(s_0,y_0\right)+\left(\epsilon^2-\left|(e^{t_0},x_0)-(e^{s_0},y_0)\right|^2\right)^{1 / 2},
$$
and observe that $\left((t_0,x_0), (s_0,y_0)\right)$ is a maximum point of the function $$\left((t,x), (s,y)\right) \rightarrow u(s,y)+ \left(\epsilon^2-|(e^{t},x)-(e^{s},y)|^2\right)^{1 / 2}-\phi(t,x).$$

First of all we show that $$\left|(e^{t_0},x_0)-(e^{s_0},y_0)\right|<\epsilon.$$
We suppose $\left|(e^{t_0},x_0)-(e^{s_0},y_0)\right|=\epsilon,$ and get a contradiction. The function $\psi$ on $[0,1]$ defined by
\eqref{4.5}
\begin{equation}\label{4.5}
\begin{aligned}
\psi(\tau)&=u\left(s_0,y_0\right)+\left(\epsilon^2-(1-\tau)^2\left|(e^{t_0},x_0)-(e^{s_0},y_0)\right|^2\right)^{1 / 2}\\&-\phi\left(t_0+\tau(s_0-t_0),x_0+\tau(y_0-x_0)\right),
\end{aligned}
\end{equation}
which has its maximum at $\tau=0$. Therefore,
$$
\frac{\psi(\tau)-\psi(0)}{\tau} \leq 0 \quad\quad\quad \text { for } 0<\tau\leq 1,
$$
we send $\tau \downarrow 0$, then $\psi'(\tau)\rightarrow \infty$ contradicts the fact.

The inequality
$$
u^\epsilon(t,x) \geq u\left(s_0,y_0\right)+\left(\epsilon^2-\left|(e^{t_0},x_0)-(e^{s_0},y_0)\right|^2\right)^{1 / 2},
$$
for
$$|(e^t,x)-(e^{t_0},x_0)|< r,$$
where $r=\epsilon-\left|(e^{t_0},x_0)-(e^{s_0},y_0)\right|$, yields that
$$
\liminf _{(t,x) \rightarrow (t_0,x_0)} u^\epsilon(t,x) \geq u^\epsilon\left(t_0,x_0\right) .
$$
This together with the upper semicontinuity of $u^{\epsilon}$ proves that $u^{\epsilon}$ is continuous at $(t_0,x_0)$.
Now the function $\psi$ can be defined near $\tau=0$ by \eqref {4.5} and attains a maximum at $\tau=0$. Hence $\psi^{\prime}(0)=0$ and so, by a simple computation, we see that
\begin{equation}\label{4.44}
(s_0,y_0)=(t_0,x_0)-\epsilon\left(|\nabla \phi(t_0,x_0)|^2+1\right)^{-1 / 2} \nabla \phi\left(t_0,x_0\right) .
\end{equation}
This formula uniquely determines the point $(s_0,y_0)$. 
In other words,
\begin{equation}\label{4.45}
u(s,y)+\left(\epsilon^2-\left|(e^{t_0},x_0)-(e^s,y)\right|^2\right)^{1 / 2}<u^{\epsilon}\left(t_0,x_0\right),
\end{equation}
for $$(s,y) \in \{(s,y)\big|\ |(e^s,y)-(e^{t_0},x_0)| < \epsilon\} \setminus{(s_0,y_0)}. $$
We fix any $\delta$ such that $\left|(e^{t_0},x_0)-(e^{s_0},y_0) \right|<\delta<\epsilon$, then the above inequality implies that there is a $\gamma>0$ such that
\begin{equation}
\begin{aligned}\label{4.46}
&\max \left\{u(s,y)+\left(\epsilon^2-|(e^t,x)-(e^s,y) |^2\right)^{1 / 2}\big|\ \delta \leq|(e^t,x)-(e^s,y) | \leq \epsilon\right\} \\<&u^\epsilon(t,x),
 \ \quad \quad\quad \ \text{for} \ (t,x) \in\{(t,x)\big|\ |(e^t,x)-(e^{t_0},x_0)|< \gamma\}.
 \end{aligned}
 \end{equation}

Indeed, if this were false, then there would be sequences $\left\{t_k,x_k\right\}$ and $\left(s_k,y_k\right\}$ of points of $\mathbb{R}^n_{+}$ such that
$$
(t_k,x_k) \rightarrow (t_0,x_0) \quad\text { as } k \rightarrow \infty \text {, }
$$
and
$$
u^\epsilon\left(t_k,x_k \right)=u\left(s_k,y_k\right)+\left(\epsilon^2-\left|(e^{t_k},x_k) -(e^{s_k},y_k)\right|^2\right)^{1 / 2},$$
where $\delta \leq\left|(e^{t_k},x_k) -(e^{s_k},y_k)\right| \leq \epsilon .
$
Passing to the limit, we find that
$$
u^\epsilon\left(t_0,x_0\right) \leq u(\bar{s},\bar{y})+\left(\epsilon^2-\left|(e^{t_0},x_0)-(e^{\bar{s}},\bar{y})\right|^2\right)^{1 / 2}, $$
for some
$$
({\bar{s}},\bar{y}) \in\{(s,y)\ \big|\ \delta\leq|(e^s,y)-(e^{t_0},x_0)|\leq\epsilon\},
$$
which contradicts \eqref{4.45} and hence proves \eqref{4.46} for some $\gamma>0$. Fix $0<\gamma<$ $\frac{1}{3}(\epsilon-\delta)$ so that \eqref{4.46} holds. Inequality \eqref{4.46} guarantees that, for $$(t,x) \in\{(s,y)\  \big|\ |(e^s,y)-(e^{t_0},x_0)|<\gamma\},$$
$$
\begin{aligned}
u^{\epsilon}(t,x)=\max \bigg\{&u(s,y)+\left(\epsilon^2-|(e^t,x)-(e^s,y)|^2\right)^{1 / 2}\ \big{|}\\&(s,y) \in\{(s,y)\ \big{|} \ |(e^s,y)-(e^{t},x)|<\delta\}\bigg\}.
\end{aligned}
$$
And hence
$$
\begin{aligned}
u^\epsilon(t,x)=\max \bigg\{&u(s,y)+\left(\epsilon^2-|(e^t,x)-(e^s,y)|^2\right)^{1 / 2}\ \big{|}\ \\& (s,y) \in\{(s,y)\ \big{|}\ | (e^s,y)-(e^{t_0},x_0)|<\delta+\gamma\}\bigg\}.
\end{aligned}
$$
Observe that we define a function
$$
f_{(s,y)}(t,x)=u(s,y)+\left(\epsilon^2-|(e^t,x)-(e^s,y)|^2\right)^{1 / 2},
$$
on $ \{(t,x)\ \big{|}\ |(e^t,x)-(e^{t_0},x_0)|<\gamma\}$, for
$$(s,y) \in\{(s,y)\ \big{|}\ |(e^s,y)-(e^{t_0},x_0)|<\delta+\gamma\}.$$
Then
by calculation, we have
$$
\begin{aligned}&\nabla_{\mathbb{B}}^2f_{(s,y)}(t,x)\\=&
    \begin{pmatrix}
        -t^2 k_1^2M^{\frac{-3}{2}}-(t^2+tk_1)M^{\frac{-1}{2}}&
        -tk_1k_2M^{\frac{-3}{2}}&\cdots& - tk_1k_nM^{\frac{-3}{2}}\\
        -tk_1k_2M^{\frac{-3}{2}}& -k_2^2M^{\frac{-3}{2}}-M^{\frac{-1}{2}}&\cdots&- k_2k_nM^{\frac{-3}{2}}\\
       \vdots &\vdots &\ddots &\vdots \\
       -tk_1k_nM^{\frac{-3}{2}}&-k_2k_nM^{\frac{-3}{2}} &\cdots&-k_n^2M^{\frac{-3}{2}}-M^{\frac{-1}{2}}
    \end{pmatrix}\\=&
     \begin{pmatrix}
        -t^2 k_1^2M^{\frac{-3}{2}}&
        -tk_1k_2M^{\frac{-3}{2}}&\cdots& - tk_1k_nM^{\frac{-3}{2}}\\
        -tk_1k_2M^{\frac{-3}{2}}& -k_2^2M^{\frac{-3}{2}}&\cdots&- k_2k_nM^{\frac{-3}{2}}\\
       \vdots &\vdots &\ddots &\vdots \\
       -tk_1k_nM^{\frac{-3}{2}}&-k_2k_nM^{\frac{-3}{2}} &\cdots&-k_n^2M^{\frac{-3}{2}}
    \end{pmatrix}\\+&
     \begin{pmatrix}
       -(t^2+tk_1)M^{\frac{-1}{2}}&
   0&\cdots&0\\
       0&-M^{\frac{-1}{2}}&\cdots&0\\
       \vdots &\vdots &\ddots &\vdots \\
      0&0 &\cdots&-M^{\frac{-1}{2}}
    \end{pmatrix}\\
    =&A+B,
    \end{aligned} $$
with $ M=\epsilon^2-|(e^t,x)-(e^s,y)|^2,\ k_1=t-s,\ k_{i+1}=x_i-y_i,\ (i=1, \cdots, n-1).$

We will simplify the matrix $\nabla_{\mathbb{B}}^2f_{(s,y)}(t,x)$ based on different situations.\\
$(1)$ $k_i=0$, $(i=1, \cdots, n)$.

We simplify that
$$
\nabla_{\mathbb{B}}^2f_{(s,y)}(t,x)=
   \begin{pmatrix}
       -t^2M^{\frac{-1}{2}}&0&\cdots&0\\
       0&-M^{\frac{-1}{2}}&\cdots&0\\
       \vdots &\vdots &\ddots &\vdots \\
      0&0 &\cdots&-M^{\frac{-1}{2}}
    \end{pmatrix}.
$$
Since $t\in (\epsilon-\gamma,1-\epsilon+\gamma) $ and $M\in (\epsilon^2-(\delta+2\gamma)^2,\epsilon^2]$, we have
$$\nabla_{\mathbb{B}}^2f_{(s,y)}(t,x)\geq - \left(\epsilon^2-(\delta+2\gamma\right)^2)^{-\frac{1}{2}}I.$$
$(2)$ $k_i\  (i=1,\cdots, n)$ are not all zero, we will discuss in two different situations.

 If $k_1\neq 0$, through the elementary row transformations, 
 A can be simplified to
$$
 \begin{pmatrix}
      -t^2 k_1^2M^{\frac{-3}{2}}&
        -tk_1k_2M^{\frac{-3}{2}}&\cdots& - tk_1k_nM^{\frac{-3}{2}}\\\\
       0&0&\cdots&0\\
       \vdots &\vdots &\ddots &\vdots \\
      0&0 &\cdots&0
    \end{pmatrix}.
  $$
Because the rank of matrix A is 1, there is only one non-zero eigenvalue of A,
 $$\lambda_{A}= tr(A)= -(t^2 k_1^2+ k_2^2+\cdots +k_n^2)M^{\frac{-3}{2}}.$$
Since $t\in (\epsilon-\gamma,1-\epsilon+\gamma) $, we have
$$
\lambda_{A}\ge -M^{\frac{-3}{2}}(k_1^2+ k_2^2+\cdots +k_n^2)\ge -\frac{(\delta+2\gamma)^2}{\left(\epsilon^2-(\delta+2\gamma\right)^2)^{\frac{3}{2}}},
$$
so
$$
A+\frac{(\delta+2\gamma)^2}{\left(\epsilon^2-(\delta+2\gamma\right)^2)^{\frac{3}{2}}}I\geq 0.
$$
Because
$$
\begin{aligned}
-(t^2+tk_1)&\ge -(1-\epsilon+\gamma)(1-\epsilon+\gamma+\delta+2\gamma)\\
&\ge -(1-\epsilon+\gamma)(1-\epsilon+3\gamma+\delta)\\
&\ge -(1-\epsilon+\gamma)(1-\epsilon+\epsilon-\delta+\delta)\\
&=-(1-\epsilon+\gamma)\\
&\ge-1,
\end{aligned}
$$
then
$$B+\frac{1}{\left(\epsilon^2-(\delta+2\gamma\right)^2)^{\frac{1}{2}}}I\geq 0,$$
so $$
\begin{aligned}
\nabla_{\mathbb{B}}^2f_{(s,y)}(t,x) &\geq -\frac{\epsilon^2}{\left(\epsilon^2-(\delta+2\gamma\right)^2)^{\frac{3}{2}}}I.
\end{aligned}
$$
If $k_1=0$, then there exists $i \neq1 $ such that $k_i\neq 0$. Because the status of $x_i\ (i=1, \cdots, n-1)$  are equivalent, let's assume that $k_2\neq0$.
Similarly, A can be simplified into the following form
$$\begin{pmatrix}
       0 &0&\cdots& 0&0\\
       0&- k_2^2M^{\frac{-3}{2}}&\cdots&- k_2k_{n-1}M^{\frac{-3}{2}}&- k_2k_nM^{\frac{-3}{2}}\\
       \vdots &\vdots &\ddots &\vdots &\vdots\\
      0&0&\cdots&0&0
      \\0&0 & \cdots &0&0
    \end{pmatrix} .$$
The rank of matrix A is $1$, there is only one non-zero eigenvalue of A,
$$
     \lambda_{A}= tr(A)= -( k_2^2+\cdots +k_n^2)M^{\frac{-3}{2}}.
$$
Since $t\in (\epsilon-\gamma,1-\epsilon+\gamma) $, we have
$$
A+\frac{(\delta+2\gamma)^2}{\left(\epsilon^2-(\delta+2\gamma\right)^2)^{\frac{3}{2}}}I\geq 0,$$
   $$B+\frac{1}{\left(\epsilon^2-(\delta+2\gamma\right)^2)^{\frac{1}{2}}}I\geq 0,$$
then
$$
\nabla_{\mathbb{B}}^2f_{(s,y)}(t,x) \geq -\frac{\epsilon^2}{\left(\epsilon^2-(\delta+2\gamma\right)^2)^{\frac{3}{2}}}I.
$$
Overall,
$$
 \nabla_{\mathbb{B}}^2f_{(s,y)}(t,x) \geq -\epsilon^2\left(\epsilon^2-(\delta+2 \gamma)^2\right)^{-3 / 2} I.
$$
Similarly,$$
\nabla^2f_{(s,y)}(t,x) \geq-\epsilon^2\left(\epsilon^2-(\delta+2 \gamma)^2\right)^{-3 / 2} I .
$$

Finally we remark that the choice of $\delta$ depends only on $\epsilon$ and $\left|\nabla \phi\left(t_0,x_0\right)\right|$ while the positive number $\gamma$ can be chosen as small as desired.
\end{proof}

 \begin{lemma}[
 \cite{HI26}]
 \label{Lemma3.1}
 Let $U$ be a bounded open subset of $\mathbb{R}^m$ and $w$ a Lipschitz continuous function on $\overline{U}$. Assume $w(y)=\mathop{\max}\limits_U w>\mathop{\max}\limits_{\partial U} w$ for some $y \in U$ and that $w$ is semi-convex on $U$. Then for any $\varepsilon>0$ there are points $p \in \mathbb{R}^m$ and $z \in U$ satisfying $|p| \leq \varepsilon$ such that the function $x \rightarrow w(x)+\langle p, x\rangle$ on $U$ attains a maximum at $z$ and has the second differential at $z$.
 \end{lemma}

 \begin{lemma}[
 \cite{HI26}]
 \label{Lemma3.2}
  For any $C>0$ the subset $K=\left\{X \in S^n\ | \ -C I \leq X \leq  C I\right\}$ of $S^n$ is compact.
 \end{lemma}

\begin{Proposition}\label{P2} Let $u$ be an $u.s.c.$ viscosity subsolution of
$$
G((t,x),u,\nabla_\mathbb{B}u,\nabla_\mathbb{B}^2u)=0 \ \ \ \ \ \ \ \ in\ \mathbb B,
$$
and
$v$ be a $l.s.c.$ viscosity supersolution of
$$
G((t,x),v,\nabla_\mathbb{B}v,\nabla_\mathbb{B}^2v)=0 \ \ \ \ \ \ \ \ in\ \mathbb B.
$$
Let $\epsilon>0$ and $\phi \in C^2(\mathbb{B})$. Set $$w((t,x),(s,y))=u^\epsilon(t,x)-v_\epsilon(s,y)$$ for $((t,x),(s,y))\in \mathbb B_\epsilon \times \mathbb B_\epsilon$. Suppose $w-\phi$ attains a maximum at some $((\overline t,\overline x),(\overline s,\overline y)\in \mathbb B_\epsilon\times \mathbb B_\epsilon$, then there exist matrices $X,Y \in S^n$ and a constant $C>0$, depending only on $\epsilon$ and $|\nabla\phi((\overline t,\overline x),(\overline s,\overline y))|$ such that
$$
-CI \le \begin{pmatrix}
X & O \\
O & Y
\end{pmatrix} \le \nabla_\mathbb{B}^2\phi((\overline t,\overline x),(\overline s,\overline y)),
$$
$$
G^\epsilon ((\overline t,\overline x),u^\epsilon(\overline t,\overline x),\nabla_{\mathbb{B}_{(t,x)}}\phi((\overline t,\overline x),(\overline s,\overline y)),X)\ge 0,
$$
and
$$
G_\epsilon ((\overline s,\overline y),v_\epsilon(\overline s,\overline y),-\nabla_{\mathbb{B}_{(s,y)}}\phi((\overline t,\overline x),(\overline s,\overline y)),-Y)\le 0.
$$
\end{Proposition}
\begin{proof}We may assume that $w-\phi$ attains a strict maximum over $\mathbb{B}_{\epsilon} \times \mathbb{B}_{\epsilon}$ at $\left((\bar{t},\bar{x}), (\bar{s},\bar{y})\right)$.
It's indeed, we can replace $\phi$ by $$
\phi((t,x), (s,y))+|(e^t,x)-(e^{\bar t},\bar{x})|^4+|(e^s,y)-(e^{\bar s},\bar{y})|^4.$$
 By Proposition \ref{Proposition3.2} we see that there is an open neighborhood $U$ of $\left((\bar{t},\bar{x}), (\bar{s},\bar{y})\right)$ and a constant $$C_0=\epsilon^2\left\{\epsilon^2-(\delta+2 \gamma)^2\right\}^{-3 / 2}$$ for which the function $$((t,x), (s,y)) \rightarrow w((t,x), (s,y)) +\frac{1}{2} C_0\left(|(e^t,x)|^2+|(e^s,y)|^2\right)$$ satisfying $
\nabla^2w \geq -C_0I$  on $U$.

 Lemma \ref{Lemma3.1} guarantees the existence of sequences $$\left\{\left((t_k,x_k), (s_k,y_k)\right)\right\} \subset U$$ and $$\left\{(p'_k,p_k)\right\},\left\{(q'_k,q_k)\right\} \subset \mathbb{R}^n$$ that satisfy the following properties (i)-(iii). \\
(i) as $k \rightarrow \infty$, $$\left((t_k,x_k), (s_k,y_k)\right) \rightarrow\left((\bar{t},\bar{x}), (\bar{s},\bar{y})\right),$$ $$(p'_k,p_k), (q'_k,q_k)\rightarrow 0.$$ \\
(ii) $w$ has the second differential at $\left((t_k,x_k), (s_k,y_k)\right) $ for $k \in \mathbf{N}$.\\
(iii) The function
\begin{equation*}
\begin{aligned}
 w\left((t,x),( s,y)\right) -\phi\left((t,x),( s,y)\right)
-\left\langle (p'_k,p_k),(t, x)\right\rangle+\left\langle (q'_k,q_k), (s,y)\right\rangle
\end{aligned}
\end{equation*}
    attains its maximum over $U$ at $\left((t_k,x_k), (s_k,y_k)\right) $ for $k \in \mathbf{N}$.

Then we get
$$\nabla_{\mathbb{B}}(w-\phi)\left((t_k,x_k), (s_k,y_k)\right) =\left(p'_k t_k,p_k,-q'_k s_k,-q_k\right),$$
$$\nabla_{\mathbb{B}}^2(w-\phi)\left((t_k,x_k),(s_k, y_k)\right) + \begin{pmatrix}
-p'_kt_k &\cdots &0&\cdots &0 \\
 \vdots&\ddots &\vdots&\ddots\\
0&\cdots&q'_ks_k &\cdots&0\\
\vdots&\ddots&\vdots&\ddots\\
0&\cdots&0&\cdots&0
\end{pmatrix}
\leq O.
$$

By Lemma \ref{Proposition3.1}, we can get
$$G^\epsilon\left((t_k,x_k), u^\epsilon ,\nabla_{\mathbb{B}}u^\epsilon, \nabla_{\mathbb{B}}^2u^\epsilon\right) \geq 0,$$
and
$$
G_\epsilon\left((s_k,y_k), v_{\epsilon}, \nabla_{\mathbb{B}} v_{\epsilon}, \nabla_{\mathbb{B}}^2 v_{\epsilon}\right) \leq0,
$$
for $k \in \mathbf{N}$. We set
$$
X_k=\nabla_{\mathbb{B}}^2 u^\epsilon\left(t_k,x_k\right) \text { and } Y_k=-\nabla_{\mathbb{B}}^2 v_{\epsilon}\left(s_k,y_k\right) \text { for } k \in \mathbf{N},
$$
by Proposition \ref{Proposition3.2}, we know there exists the constant $C_1 $ such that
$$
\begin{gathered}
-C_1I \leq\left(\begin{array}{cc}
X_k & O \\
O & Y_k
\end{array}\right) \leq \nabla_{\mathbb{B}}^2 \phi\left((t_k,x_k),(s_k, y_k)\right)-
\begin{pmatrix}
-p'_kt_k &\cdots &0&\cdots &0 \\
 \vdots&\ddots &\vdots&\ddots\\
0&\cdots&q'_ks_k &\cdots&0\\
\vdots&\ddots&\vdots&\ddots\\
0&\cdots&0&\cdots&0
\end{pmatrix},
\end{gathered}
$$

$$G^{\epsilon} \left((t_k,x_k), u^{\epsilon}, \nabla_{\mathbb{B}_{(t,x)}} \phi\left((t_k,x_k), (s_k,y_k)\right)+(p'_kt_k,p_k), X_k\right) \geq 0,$$
and
$$G_{\epsilon} \left((s_k,y_k), v_{\epsilon}, -\nabla_{\mathbb{B}_{(s,y)}} \phi\left((t_k,x_k), (s_k,y_k)\right)+(q'_ks_k,q_k), -Y_k\right) \leq 0,$$
for $k \in \mathbf{N}$.
Since $-C I \leq X_k, Y_k \leq CI$ for some $C>0$ and all $k \in \mathbf{N}$, we see from Lemma \ref{Lemma3.2} that there is an increasing sequence $\left\{k_j\right\} \subset \mathbf{N}$ and matrices $X$, $Y \in S^n$ such that $X_{k_j} \rightarrow X, Y_{k_j} \rightarrow Y$ in $S^n$ as $j \rightarrow \infty$. Moreover, 
the  $w$ is continue on $U$. Thus, we send $k=k_j \rightarrow \infty$ to obtain that
$$
\begin{gathered}
-CI \leq\left(\begin{array}{ll}
X & O \\
O & Y
\end{array}\right) \leq \nabla_{\mathbb{B}}^2 \phi((\bar{t},\bar{x}), (\bar{s},\bar{y})),
\end{gathered}
$$
$$G^{\epsilon} \left((\bar{t},\bar{x}), u^{\epsilon}, \nabla_{\mathbb{B}_{(t,x)}} \phi\left((\bar{t},\bar{x}), (\bar{s},\bar{y})\right), X\right) \geq 0,$$
and
$$G_{\epsilon} \left((\bar{s},\bar{y}), v_{\epsilon}, -\nabla_{\mathbb{B}_{(s,y)}} \phi\left((\bar{t},\bar{x}), (\bar{s},\bar{y})\right), -Y\right) \leq 0.$$
\end{proof}

\begin{proof}[\textbf{Proof of Theorem \ref{A}}]
Since $u\in USC(\overline{\mathbb{B}})$, and
$ \mathop{\lim}\limits_{t \to 0}\sup u$ exists,
so we set $u \le M_u$. Since $v\in LSC(\overline{\mathbb{B}})$, $ \mathop{\lim}\limits_{t \to 0}\inf v$ exists, and $v$ is bounded from above, so we assume
 $m_v\le v\le M_v$.

 Next we will proceed with this proof in two steps.\\
\noindent \textbf{Step 1.}

  We claim that if $v$ is  a viscosity supersolution (resp. subsolution) of equation \eqref{H3}, the fact that $z$ is the viscosity supersolution (resp. subsolution) of equation \eqref{H9}
\begin{equation}\label{H9}
\begin{aligned}
    &|\nabla_\mathbb{B}z|^{p-2}tr\big[\big(I+(p-2)\frac{\nabla_\mathbb{B}z(\nabla_\mathbb{B}z)^T}{|\nabla_\mathbb{B}z|^2}\big)\nabla_\mathbb{B}^2z\big]+\frac{\psi''(z)}{\psi'(z)}(p-1)|\nabla_\mathbb{B}z|^p\\
    &+(n-p)|\nabla_\mathbb{B}z|^{p-2}(t\partial_t z)=\frac{f(t,x)t^{p}}{\psi'(z)^{p-1}},     \ \ \ \ \ \ \ \ \ \    (t,x)\in \mathbb B,
\end{aligned}
\end{equation}
where $v(t,x)=\psi(z(t,x))$, $\psi\in C^2(\mathbb{R})$ is monotonically increasing and $\psi'\not=0$.
We take the viscosity supersolution as an example for the proof.

For all $\xi\in C^2(\mathbb{B})$, $z-\xi$  reaches a local minimum value at $(t_0,x_0)$. Let $ c_0=(z-\xi)(t_0,x_0)$, and $\tilde{\xi}=\xi +c_0$, then $\psi(z)-\psi (\tilde{\xi}) $ reaches a local minimum  at $(t_0,x_0)$, and $\psi(z(t_0,x_0))=\psi(\tilde{\xi}(t_0,x_0))$.
Since $u$ is the viscosity supersolution of \eqref{H3}, then
\begin{equation*}
\begin{aligned}
    &|\nabla_\mathbb{B}\psi(\tilde{\xi})|^{p-2}tr\big[\big(I+(p-2)\frac{\nabla_\mathbb{B}\psi(\tilde{\xi})(\nabla_\mathbb{B}\psi(\tilde{\xi}))^T}{|\nabla_\mathbb{B}\psi(\tilde{\xi})|^2}\big)\nabla_\mathbb{B}^2\psi(\tilde{\xi})\big]\\
    &+(n-p)|\nabla_\mathbb{B}\psi(\tilde{\xi})|^{p-2}(t\partial_t \psi(\tilde{\xi}))\le f(t_0,x_0)t_0^{p},   \ \ \ \ \ \ \ \ \ \   (t_0,x_0)\in \mathbb B.
\end{aligned}
\end{equation*}
And since
$$\nabla_\mathbb{B}\psi(\tilde{\xi }(t,x))=\psi'(\tilde{\xi})\ \nabla_\mathbb{B}\tilde{\xi}=\psi'(\tilde{\xi})\ \nabla_\mathbb{B}\xi,$$
\begin{equation*}
\begin{aligned}
\nabla_\mathbb{B}^2\psi(\tilde{\xi}(t,x))&=\psi'(\tilde{\xi})\ \nabla_\mathbb{B}^2\tilde{\xi}+\psi''(\tilde{\xi})\ \nabla_\mathbb{B}\tilde{\xi} (\nabla_\mathbb{B}\tilde{\xi})^T\\
&=\psi'(\tilde{\xi})\ \nabla_\mathbb{B}^2\xi+\psi''(\tilde{\xi})\ \nabla_\mathbb{B}\xi (\nabla_\mathbb{B}\xi)^T,
\end{aligned}
\end{equation*}
substitute them into the above equation, we have
\begin{equation*}
\begin{aligned}
     &|\psi'(\tilde{\xi})\nabla_\mathbb{B}\xi|^{p-2}tr\big[\big(I+(p-2)\frac{\nabla_\mathbb{B}\xi(\nabla_\mathbb{B}\xi)^T}{|\nabla_\mathbb{B}\xi|^2}\big)\big(\psi'(\tilde{\xi})\nabla_\mathbb{B}^2\xi+\psi''(\tilde{\xi})\nabla_\mathbb{B}\xi(\nabla_\mathbb{B}\xi)^T\big)\big]\\
     &+(n-p)\psi'(\tilde{\xi})(t\partial_t \xi)|\psi'(\tilde{\xi})\nabla_\mathbb{B}\xi|^{p-2}\le f(t_0,x_0)t_0^{p},     \ \ \ \ \ \ \ \ \ \   (t_0,x_0)\in \mathbb B.
\end{aligned}
\end{equation*}
Finally, we get,
\begin{equation*}
\begin{aligned}
     &|\nabla_\mathbb{B}\xi|^{p-2}tr\big[\big(I+(p-2)\frac{\nabla_\mathbb{B}\xi(\nabla_\mathbb{B}\xi)^T}{|\nabla_\mathbb{B}\xi|^2}\big)\nabla_\mathbb{B}^2\xi\big]+(p-1)\frac{\psi''(\tilde{\xi})}{\psi'(\tilde{\xi})}|\nabla_\mathbb{B}\xi|^{p}\\
     &+(n-p)(t\partial_t \xi)|\nabla_\mathbb{B}\xi|^{p-2}\le \frac{f(t_0,x_0)t_0^{p}}{\psi'(\tilde{\xi})^{p-1}},     \ \ \ \ \ \ \ \ \ \   (t_0,x_0)\in \mathbb B.
\end{aligned}
\end{equation*}

According to the Definition $\ref{D1}$ , we prove the above claim that $z$ is the viscosity supersolution (resp. subsolution) of equation \eqref{H9} .

Let $$ M=\max\{|M_u|, |M_v|,|m_v|\},$$
if $M=0$, we can get $$u\le0,v=0,$$ then Theorem \ref{A} holds.

If $M \neq 0$, let
$$\psi(s)=K\int_{0}^{s} e^{-r}\, dr,$$ where $K=2M<+\infty$. We know that $$u,|v| \le M.$$
Without loss of generality, let
$$u=\psi(z_1),v=\psi(z_2),$$
 then we have
  $$z_1\in (-\infty,ln2],z_2\in[ln\frac{2}{3},ln2],$$

 $$\psi'(s)=Ke^{-s},\psi''(s)=-Ke^{-s}.$$
So $\eqref{H9}$ can be simplified to the following expression
\begin{equation}\label{H10}
\begin{aligned}
     &|\nabla_\mathbb{B}z|^{p-2}tr\big[\big(I+(p-2)\frac{\nabla_\mathbb{B}z(\nabla_\mathbb{B}z)^T}{|\nabla_\mathbb{B}z|^2}\big)\nabla_\mathbb{B}^2z\big]-(p-1)|\nabla_\mathbb{B}z|^p\\
     &+(n-p)|\nabla_\mathbb{B}z|^{p-2}(t\partial_t z)-\frac{f(t,x)t^{p}}{K^{p-1}}e^{z(p-1)}=0.
\end{aligned}
\end{equation}

\noindent \textbf{Step 2.}

We claim that
    if $t^pf(t,x)\geq \omega>0$, for any $(t,x)\in \mathbb B$, then the viscosity supersolution (resp.subsolution) of the equation \eqref{H10} satisfies the comparison principle, that is
\begin{equation}\label{H11}
    \sup\limits_{\mathbb B} z_1-z_2\le \sup\limits_{\partial \mathbb B}(z_1-z_2)^+=0,
\end{equation}
where $z_1$ is the viscosity subsolution of \eqref{H10}, $z_2$ is the viscosity supersolution of \eqref{H10}. Here we proof it by a contradiction.

If$$\eqref{H11}\text{\ is\  not\  satisfied,\  then}
\sup\limits_{\mathbb B} z_1-z_2> \sup\limits_{\partial \mathbb B}(z_1-z_2)^+=0 .$$
We note that
\begin{equation}\label{eq:33}
M_{\alpha}=\sup\limits_{{\mathbb  B}\times {\mathbb  B}}z_1(t,x)-z_2(s,y)-\frac{\alpha}{2}\big((\ln\frac{t}{s})^2+|x-y|^2\big),\end{equation} which $\alpha >0$. Since  $z_1\in (-\infty,ln2]$ and  $z_2\in[ln\frac{2}{3},ln2] $, so $M_{\alpha}$ exists, and $$M_{\alpha}\ge \sup\limits_{\mathbb B} z_1-z_2=\delta>0.$$
We can easily notice that when $(t,x)\ \mbox{or}\ (s,y)\to \partial{\mathbb{B}}$,
$$ \mathop{\lim}\limits_{\alpha \to \infty }z_1(t,x)-z_2(s,y)-\frac{\alpha}{2}\left((\ln\frac{t}{s})^2+(|x-y|)^2\right)\leq 0,$$
and when $t\ \mbox{or}\ s\to 0$,
$$ \mathop{\lim}\limits_{\alpha \to \infty }z_1(t,x)-z_2(s,y)-\frac{\alpha}{2}\left((\ln\frac{t}{s})^2+(|x-y|)^2\right)\leq 0.$$
So when $\alpha$ is  large enough, we can find  $((t_{\alpha},x_{\alpha}), (s_{\alpha},y_{\alpha}))\in \mathbb{B}\times\mathbb{B}$ such that $$M_{\alpha}=z_1(t_{\alpha},x_{\alpha})-z_2(s_{\alpha},y_{\alpha})-\frac{\alpha}{2}\big((\ln \frac {t_{\alpha}}{s_{\alpha}})^2+|x_{\alpha}-y_{\alpha}|^2\big),$$
and when $\alpha\to \infty,$ there exists a subsequence of $((t_{\alpha},x_{\alpha}), (s_{\alpha},y_{\alpha}))$, which we still denote itself, converging to some point  $((t_{\infty},x_{\infty}),\left(t_{\infty},x_{\infty})\right)\in \mathbb{B}\times\mathbb{B}.$
Similarly, we note that $$ M_{\alpha,\epsilon}=\sup\limits_{{\mathbb B}_{\epsilon}\times {\mathbb  B}_{\epsilon}}z^{\epsilon}_1(t,x)-z_{2{\epsilon}}(s,y)-\frac{\alpha}{2}\big((\ln\frac{t}{s})^2+|x-y|^2\big),$$
and choosing  $((t_{\alpha,\epsilon},x_{\alpha,\epsilon}), (s_{\alpha,\epsilon},y_{\alpha,\epsilon}))$ such that
$$M_{\alpha,\epsilon}= z_1^{\epsilon}(t_{\alpha,\epsilon},x_{\alpha,\epsilon})-z_{2\epsilon}(s_{\alpha,\epsilon},y_{\alpha,\epsilon})-\frac{\alpha}{2}\big((\ln\frac{t_{\alpha,\epsilon}}{s_{\alpha,\epsilon}})^2+|x_{\alpha,\epsilon}-y_{\alpha,\epsilon}|^2\big).$$
When   $\epsilon\to 0,$  there exists a subsequence of $((t_{\alpha,\epsilon},x_{\alpha,\epsilon}), (s_{\alpha,\epsilon},y_{\alpha,\epsilon}))$, which we still denote itself, converging to some point  $((t_{\alpha,0},x_{\alpha,0}),\left(t_{\alpha,0},x_{\alpha,0})\right).$
Since $$
\begin{aligned}M_{\alpha}=&\lim \limits_{\epsilon \to 0}M_{\alpha,\epsilon}\\ \leq & z_1(t_{\alpha,0},x_{\alpha,0})-z_2(s_{\alpha,0},y_{\alpha,0})-\frac{\alpha}{2}\big((\ln \frac {t_{\alpha,0}}{s_{\alpha,0}})^2+|x_{\alpha,0}-y_{\alpha,0}|^2\big),
\end{aligned}$$
and  when $\alpha$ is  large enough,
$$
\begin{aligned} z_1(t_{\alpha,0},x_{\alpha,0})-z_2(s_{\alpha,0},y_{\alpha,0})-\frac{\alpha}{2}\big((\ln \frac {t_{\alpha,0}}{s_{\alpha,0}})^2+|x_{\alpha,0}-y_{\alpha,0}|^2\big)\leq M_{\alpha}
\end{aligned}.$$
So when $\alpha$ is sufficiently large, we have
 $$ M_{\alpha} = z_1(t_{\alpha,0},x_{\alpha,0})-z_2(s_{\alpha,0},y_{\alpha,0})-\frac{\alpha}{2}\big((\ln \frac {t_{\alpha,0}}{s_{\alpha,0}})^2+|x_{\alpha,0}-y_{\alpha,0}|^2\big).$$
 We can assume
  $$((t_{\alpha,0},x_{\alpha,0}),(t_{\alpha,0},x_{\alpha,0}))= \left((t_{\alpha},x_{\alpha}),(t_{\alpha},x_{\alpha})\right) . $$
So when $\alpha$ is sufficiently large, and $\epsilon$ is sufficiently small, we have $$\big((t_{\alpha,\epsilon},x_{\alpha,\epsilon}),(s_{\alpha,\epsilon},y_{\alpha,\epsilon})\big)\in \mathbb B_\epsilon \times \mathbb B_\epsilon,$$ 
and
$$
\lim \limits_{\epsilon \to 0}\big((t_{\alpha,\epsilon},x_{\alpha,\epsilon}),(s_{\alpha,\epsilon},y_{\alpha,\epsilon})\big)
=\left((t_\alpha,x_{\alpha }),(s_\alpha,y_{\alpha})\right).
$$
For convenience, we simplify \eqref{H10} by noting that
\begin{equation*}
 \begin{aligned}
G&=|\nabla_\mathbb{B}z|^{p-2}tr\big[\big(I+(p-2)\frac{\nabla_\mathbb{B}z(\nabla_\mathbb{B}z)^T}{|\nabla_\mathbb{B}z|^2}\big)\nabla_\mathbb{B}^2z\big]-(p-1)|\nabla_\mathbb{B}z|^p\\
     &+(n-p)|\nabla_\mathbb{B}z|^{p-2}(t\partial_t z)-\frac{f(t,x)t^{p}}{K^{p-1}}e^{z(p-1)},
    \end{aligned}
    \end{equation*}
    and
\begin{equation*}
    \begin{aligned}
J&=G((t_{\alpha},x_{\alpha}),z_1(t_{\alpha},x_{\alpha}),q,X)-G((t_{\alpha},x_{\alpha}),z_2(s_{\alpha},y_{\alpha}),q,X)\\
&=\frac{t_{\alpha}^pf(t_{\alpha},x_{\alpha})}{K^{p-1}}\big(e^{(p-1)z_2(s_{\alpha},y_{\alpha})}-e^{(p-1)z_1(t_{\alpha},x_{\alpha})}\big)\\
&=\frac{t_{\alpha}^pf(t_{\alpha},x_{\alpha})}{K^{p-1}}(p-1)e^{(p-1)\widetilde z}\big(z_2(s_{\alpha},y_{\alpha})-z_1(t_{\alpha},x_{\alpha})\big),\\
\end{aligned}
\end{equation*}
with $\tilde{z}\in [z_2,z_1]$. Since $t^p_{\alpha}f(t_{\alpha},x_{\alpha})\ge \omega>0$, and   $$z_1\in (-\infty,ln2],z_2\in[ln\frac{2}{3},ln2],$$ so, when $\alpha\to \infty$, we can get $$J<-(p-1)e^{(p-1)\ln{\frac{2}{3}}}\delta\omega.$$
On the other hand,
\begin{equation*}
    \begin{aligned}
J&=G((t_{\alpha},x_{\alpha}),z_1(t_{\alpha},x_{\alpha}),q,X)-G((t_{\alpha},x_{\alpha}),z_2(s_{\alpha},y_{\alpha}),q,X)\\
&=G((t_{\alpha},x_{\alpha}),z_1(t_{\alpha},x_{\alpha}),q,X)-G^{\epsilon}((t_{\alpha,\epsilon},x_{\alpha,\epsilon}),z_1^{\epsilon}(t_{\alpha,\epsilon},x_{\alpha,\epsilon}),q,X)\\
&+G^{\epsilon}((t_{\alpha,\epsilon},x_{\alpha,\epsilon}),z_1^{\epsilon}(t_{\alpha,\epsilon},x_{\alpha,\epsilon}),q,X)-G_{\epsilon}((s_{\alpha,\epsilon},y_{\alpha,\epsilon}),z_{2\epsilon}(s_{\alpha,\epsilon},y_{\alpha,\epsilon}),q,-Y)\\
&+G_{\epsilon}((s_{\alpha,\epsilon},y_{\alpha,\epsilon}),z_{2\epsilon}(s_{\alpha,\epsilon},y_{\alpha,\epsilon}),q,-Y)-G_{\epsilon}((s_{\alpha,\epsilon},y_{\alpha,\epsilon}),z_{2\epsilon}(s_{\alpha,\epsilon},y_{\alpha,\epsilon}),q,X)\\
&+G_{\epsilon}((s_{\alpha,\epsilon},y_{\alpha,\epsilon}),z_{2\epsilon}(s_{\alpha,\epsilon},y_{\alpha,\epsilon}),q,X)-G((s_{\alpha},y_{\alpha}),z_{2}(s_{\alpha},y_{\alpha}),q,X)\\
&+G((s_{\alpha},y_{\alpha}),z_{2}(s_{\alpha},y_{\alpha}),q,X)-G((t_{\alpha},x_{\alpha}),z_2(s_{\alpha},y_{\alpha}),q,X)\\
&=J_1+J_2+J_3+J_4+J_5.
\end{aligned}
\end{equation*}
Choosing $X,\ Y$ like Proposition $\ref{P2}$, and $q=\alpha(\ln\frac{t_{\alpha,\epsilon}}{s_{\alpha,\epsilon}},x_{\alpha,\epsilon}-y_{\alpha,\epsilon})$, so we can illustrate that
$$ \begin{pmatrix}
        X&0\\
        0&Y
    \end{pmatrix}\leq \frac{\alpha}{2}
    \begin{pmatrix}
        2I&-2I\\
        -2I&2I\\
    \end{pmatrix}.
$$

Then $$X+Y\leq 0,\ J_3\geq 0,\ J_2\geq 0. $$
Let $\epsilon \to 0$, then $J_1,\ J_4\to 0$,
let $\alpha \to\infty$, we can find
$$J_5=\big(\frac{t_{\alpha}^pf(t_{\alpha},x_{\alpha})}{K^{p-1}}-\frac{s_{\alpha}^pf(s_{\alpha},y_{\alpha})}{K^{p-1}}\big)e^{(p-1)z_2(s_{\alpha},y_{\alpha})}\to 0,$$
so $J \geq 0$, we get a contradiction. 
So we get \eqref{H11}. Since $\psi$ monotonically increasing, then \begin{equation}
    \sup\limits_{\mathbb B} \psi(z_1)-\psi(z_2)\le 0,
\end{equation}
that is,
$u\leq v $ in $\mathbb{B}$.
\end{proof}
\begin{corollary}\label{corollary 1}
Let $u\in USC(\overline{\mathbb B})$ be  a viscosity subsolution of equation \eqref{eq:11} and $\mathop{\lim}\limits_{t \to 0}\sup u$ exists; let $v\in LSC(\overline{\mathbb B})$ be a viscosity supersolution of equation \eqref{eq:11} and $\mathop{\lim}\limits_{t \to 0}\inf v$ exists. In addition, $\mathop{\lim}\limits_{t \to 0}\sup u\leq \mathop{\lim}\limits_{t \to 0}\inf v$,
    and $u(t,x)$ is bounded from below. If $f\in C(\overline{\mathbb B})$ and there exists constant $\omega>0$ such that $t^pf(t,x)\leq -\omega$ for all $(t,x)\in \mathbb B$. Then we can deduce $u\le v$ in $\mathbb B$ when $u\le v$ on $\partial \mathbb B$.
\end{corollary}
\begin{proof}
   May as well assume $m_u\le u \le M_u$, $v\ge m_v$,
we can select $$\psi=K\int_{0}^{s} e^r\, dr,\ K=2\max\{|m_u|, |M_u|, |M_v|\}.$$
\end{proof}
\begin{corollary}\label{corollary 2}
Let $u\in USC(\overline{\mathbb B})$ be  a viscosity subsolution of equation \eqref{eq:11} and $\mathop{\lim}\limits_{t \to 0}\sup u$ exists; let $v\in LSC(\overline{\mathbb B})$ be a viscosity supersolution of equation \eqref{eq:11} and $\mathop{\lim}\limits_{t \to 0}\inf v$ exists. In addition, $\mathop{\lim}\limits_{t \to 0}\sup u\leq \mathop{\lim}\limits_{t \to 0}\inf v$, $u(t,x)$ is bounded from below and $v(t,x)$ is bounded from above. If $f\in C(\overline{\mathbb B})$ and there exists constant $\omega>0$ such that $|t^pf(t,x)|\geq \omega $ for all $(t,x)\in \mathbb B$. Then we can deduce $u\le v$ in $\mathbb B$ when $u\le v$ on $\partial \mathbb B$.
\end{corollary}
\section{\textbf{The existence of weak solutions}}
In this section, based on the work in  \cite{PPJ},  \cite{VP} and \cite{MO}, we explore the relationship between weak solutions and viscosity solutions of \eqref{eq:11}, and give the existence of weak solutions. Different from \cite{MO}, we  establish the weighted Sobolev spaces and study under these spaces.

\begin{definition}[Infimal convolution]
We define the infimal convolution of a function u as
$$ u_{\epsilon}(t,x):=\mathop{\inf }\limits_{(s,y)\in\mathbb{B}}\left(u(s,y)+\frac{|(e^t,x)-(e^s,y)|^2}{2\epsilon}\right)$$

\end{definition}

\begin{lemma}\label{Lemma2.1}
Assume $u:\mathbb{B }\to \mathbb{R}$ and $u_{\epsilon}$ is the infimal convolution of $u$.

$(1)$ $u_{\epsilon}(t,x)$ is an increasing sequence of semiconvex functions   in $\mathbb{B}$ about $(e^t,x)$ ;

$(2)$ 
 $u_{\epsilon} $ is locally Lipschitz continuous  in the interior of $\mathbb{B}$ about $(e^t,x)$ ;

$(3)$ if $u$ is lower semicontinuous in $\mathbb{B}$, $u_{\epsilon}$ converges pointwise to u;

$(4)$ if $u$ is  bounded, $u_{\epsilon}$ can be written as
$$ u_{\epsilon}(t,x):=\mathop{\inf }\limits_{(s,y)\in { B_{r(\epsilon)}(t,x)\bigcap\mathbb{B}}}\left(u(s,y)+\frac{|(e^t,x)-(e^s,y)|^2}{2\epsilon}\right)$$
for $r(\epsilon)=2\sqrt{||u||_{L^{\infty}(\mathbb{B})}\epsilon}$.
\end{lemma}
For these and further properties, see \cite{PPJ}.
%
%
%
\begin{lemma}\label{Lemma5.2}
Suppose that $u:\mathbb B\to\mathbb{R}$ is bounded in $\mathbb B$ and $f\in C(\mathbb{B})$, if $u$ is a viscosity supersolution to \eqref{eq:11} in $\mathbb B$, then $u_{\epsilon}$ satisfies
$$div_ \mathbb{B}(|\nabla_ \mathbb{B}u_{\epsilon}|^{p-2}\nabla_ \mathbb{B}u_{\epsilon})
     +(n-p)|\nabla_ \mathbb{B}u_{\epsilon}|^{p-2}(t\partial_t u_{\epsilon})\leq (t^pf(t,x))_{\epsilon} $$
     a.e. in $$\mathbb B_{r(\epsilon)}=\{(t,x)\in  \mathbb B\ \big{|}\ |(e^t,x)-(e^s,y)|^2>{r(\epsilon)}^2,\ \text{for all}\ (s,y)\in \partial  \mathbb B \cup \{0\}\times \partial X\}$$ where
     $$  (t^pf(t,x))_{\epsilon}=\mathop{\sup }\limits_{(s,y)\in { B_{r(\epsilon)}(t,x)}}s^pf(s,y).$$

\end{lemma}
\begin{lemma}\label{Lemma5.3}
 If $u$ is a viscosity supersolution to \eqref{eq:11}, then  the following inequality holds  for any non-negative $\psi  $  satisfying $\operatorname{supp} \psi \subset \mathbb{B}_{r(\epsilon)}$.
  $$
\int_{\mathbb{B}}\left|\nabla_ \mathbb{B} {u}_{\epsilon}\right|^{p-2}\nabla_ \mathbb{B} {u}_{\epsilon}\cdot \nabla_ \mathbb{B}\psi \frac{dt}{t}dx \geq \int_{\mathbb{B}}-\left(div_ \mathbb{B}(|\nabla_ \mathbb{B}{u}_{\epsilon}|^{p-2}\nabla_ \mathbb{B}{u}_{\epsilon})\right)\psi \frac{dt}{t}dx
$$
\end{lemma}
\begin{proof}
 Let $\varphi={u}_{\epsilon}-\frac{1}{2 \epsilon}|(t,x)|^2$ and  $\varphi_j$ be a sequence of smooth concave functions converging to $\varphi$, obtained via standard mollification, and  let ${u}_{\epsilon, j}=\varphi_j+\frac{1}{2 \epsilon}|(t,x)|^2$. Integration by parts gives
$$
\int_{\mathbb{B}}\left|\nabla_ \mathbb{B} {u}_{\epsilon,j}\right|^{p-2} \nabla_ \mathbb{B} {u}_{\epsilon,j} \cdot \nabla_ \mathbb{B}\psi \frac{dt}{t}dx = \int_{\mathbb{B}}-\left(div_ \mathbb{B}(|\nabla_ \mathbb{B}{u}_{\epsilon,j}|^{p-2}\nabla_ \mathbb{B}{u}_{\epsilon,j})\right)\psi \frac{dt}{t}dx
$$

Since ${u}_{\epsilon}$ is locally Lipschitz continuous, we clearly have
$$
\lim _{j \rightarrow \infty} \int_{\mathbb{B}}\left|\nabla_ \mathbb{B} {u}_{\epsilon,j}\right|^{p-2} \nabla_ \mathbb{B}{u}_{\epsilon,j} \cdot \nabla_ \mathbb{B}\psi \frac{dt}{t}dx =\int_{\mathbb{B}}\left|\nabla_ \mathbb{B} {u}_{\epsilon}\right|^{p-2}\nabla_ \mathbb{B} {u}_{\epsilon} \cdot \nabla_ \mathbb{B}\psi \frac{dt}{t}dx .
$$

On the other hand, since $D^2{u}_{\epsilon, j} \leq \frac{1}{\epsilon} I$ and $D {u}_{\varepsilon, j}$ is locally bounded, we have
$$
-\left(div_ \mathbb{B}(|\nabla_ \mathbb{B}{u}_{\epsilon,j}|^{p-2}\nabla_ \mathbb{B}{u}_{\epsilon,j})\right)\geq -C
$$
in the support of $\psi$. Thus, by Fatou's lemma,
$$
\begin{aligned}
&\liminf _{j \rightarrow \infty}  \int_{\mathbb{B}}-\left(div_ \mathbb{B}(|\nabla_ \mathbb{B}{u}_{\epsilon,j}|^{p-2}\nabla_ \mathbb{B}{u}_{\epsilon,j})\right)\psi \frac{dt}{t}dx\\ \geq& \int_{\mathbb{B}} \liminf _{j \rightarrow \infty}-\left(div_ \mathbb{B}(|\nabla_ \mathbb{B}{u}_{\epsilon,j}|^{p-2}\nabla_ \mathbb{B}{u}_{\epsilon,j})\right)\psi \frac{dt}{t}dx.
\end{aligned}
$$

 since $D^2 \varphi_j(t,x) \rightarrow D^2 \varphi(t,x)$ and  $D \varphi_j(t,x) \rightarrow D\varphi(t,x)$ for a.e. $(t,x)$, and thus
$$
\liminf _{j \rightarrow \infty}-\left(div_ \mathbb{B}(|\nabla_ \mathbb{B}{u}_{\epsilon,j}|^{p-2}\nabla_ \mathbb{B}{u}_{\epsilon,j})\right)=-\left(div_ \mathbb{B}(|\nabla_ \mathbb{B}{u}_{\epsilon}|^{p-2}\nabla_ \mathbb{B}{u}_{\epsilon})\right)
$$
for a.e. $(t,x)$. Putting everything together, we have

$$
\begin{aligned}
\int_{\mathbb{B}}\left|\nabla_ \mathbb{B} {u}_{\epsilon}\right|^{p-2}\nabla_ \mathbb{B} {u}_{\epsilon} \cdot \nabla_ \mathbb{B}\psi \frac{dt}{t}dx & =\int_{\mathbb{B}}-\left(div_ \mathbb{B}(|\nabla_ \mathbb{B}{u}_{\epsilon,j}|^{p-2}\nabla_ \mathbb{B}{u}_{\epsilon,j})\right)\psi \frac{dt}{t}dx\\
& \geq\liminf _{j \rightarrow \infty}  \int_{\mathbb{B}}-\left(div_ \mathbb{B}(|\nabla_ \mathbb{B}{u}_{\epsilon,j}|^{p-2}\nabla_ \mathbb{B}{u}_{\epsilon,j})\right)\psi \frac{dt}{t}dx \\
& \geq \int_{\mathbb{B}} \liminf _{j \rightarrow \infty}-\left(div_ \mathbb{B}(|\nabla_ \mathbb{B}{u}_{\epsilon,j}|^{p-2}\nabla_ \mathbb{B}{u}_{\epsilon,j})\right)\psi \frac{dt}{t}dx \\
& = \int_{\mathbb{B}}-\left(div_ \mathbb{B}(|\nabla_ \mathbb{B}{u}_{\epsilon}|^{p-2}\nabla_ \mathbb{B}{u}_{\epsilon})\right)\psi \frac{dt}{t}dx
\end{aligned}
$$
as desired.
\end{proof}

\begin{proof}[\textbf{Proof of Theorem \ref{Theorem 1.5}}]
By the Lemma \ref{Lemma5.2} and Lemma \ref{Lemma5.3}, we can obtain
 \begin{equation}\label{1.99}
  \begin{aligned}
&\int_{\mathbb{B}}\left|\nabla_ \mathbb{B} {u}_{\epsilon}\right|^{p-2}\nabla_ \mathbb{B} {u}_{\epsilon}\cdot \nabla_ \mathbb{B}\psi \frac{dt}{t}dx \\ \geq&\int_{\mathbb{B}} \left((-t^pf(t,x))_{\epsilon}+(n-p)|\nabla_ \mathbb{B}{u}_{\epsilon}|^{p-2}(t\partial_t {u}_{\epsilon})\right)\psi\frac{dt}{t}dx
\end{aligned}
\end{equation}
for any non-negative $\psi  $  satisfying $\operatorname{supp} \psi \subset \mathbb{B}_{r(\varepsilon)}$. Since ${u}_{\epsilon}\in \mathbb{H}_{p,loc}^{1,\frac{n}{p}}(\mathbb{B}) $ and ${\left(t(1-t)d(x,\partial X)\right)}^p\psi\in \mathbb{H}_{p,0}^{1,\frac{n}{p}}(\mathbb{B})$, so for any non-negative $\psi\in \mathcal{C}_{0}^{\infty}(\mathbb{B}_{r(\epsilon)}) $, the following inequality holds
\begin{equation}\label{2.66}
\begin{aligned}
&\int_{\mathbb{B}}{\left(t(1-t)d(x,\partial X)\right)}^p \left|\nabla_ \mathbb{B} {u}_{\epsilon}\right|^{p-2}\nabla_ \mathbb{B} {u}_{\epsilon} \cdot \nabla_ \mathbb{B}\psi \frac{dt}{t}dx \\ \geq&\int_{\mathbb{B}} {\left(t(1-t)d(x,\partial X)\right)}^p \left((-t^pf(t,x))_{\epsilon}+(n-p)|\nabla_ \mathbb{B}{u}_{\epsilon}|^{p-2}(t\partial_t {u}_{\epsilon})\right)\psi\frac{dt}{t}dx\\+&
\int_{\mathbb{B}} \psi\left|\nabla_ \mathbb{B} {u}_{\epsilon}\right|^{p-2}\nabla_ \mathbb{B} {u}_{\epsilon} \cdot \nabla_ \mathbb{B}{\left(t(1-t)d(x,\partial X)\right)}^p \frac{dt}{t}dx.
\end{aligned}
\end{equation}
We set $\xi\in \mathcal{C}_{0}^{\infty}(\mathbb{B}_{r(\epsilon)})$, $\xi=1$ in $\mathbb{B}_{3r(\epsilon)}$ and $\xi=0$ in $\mathbb{B}\setminus\mathbb{B}_{2r(\epsilon)}$. Let $\tilde{\psi}=(\mathop{\sup }\limits_{\mathbb{B}}{u}_{\epsilon}-{u}_{\epsilon} ){\xi}^{p}$. Though Lemma \ref{Lemma2.1}, we easily see $\tilde{\psi}\in \mathbb{H}_{p,0}^{1,\frac{n-p}{p}}(\mathbb{B}_{r(\epsilon)})$ as $\epsilon\to 0$, and by approximation we get
\begin{equation}\label{2.7}
\begin{aligned}
&\int_{\mathbb{B}}{\left(t(1-t)d(x,\partial X)\right)}^p \left|\nabla_ \mathbb{B} {u}_{\epsilon}\right|^{p-2}\nabla_ \mathbb{B} {u}_{\epsilon} \cdot \nabla_ \mathbb{B}\tilde{\psi} \frac{dt}{t}dx \\ \geq&\int_{\mathbb{B}} {\left(t(1-t)d(x,\partial X)\right)}^p \left((-t^pf(t,x))_{\epsilon}+(n-p)|\nabla_ \mathbb{B}{u}_{\epsilon}|^{p-2}(t\partial_t {u}_{\epsilon})\right)\tilde{\psi}\frac{dt}{t}dx\\+&
\int_{\mathbb{B}} \tilde{\psi}\left|\nabla_ \mathbb{B} {u}_{\epsilon}\right|^{p-2}\nabla_ \mathbb{B} {u}_{\epsilon} \cdot \nabla_ \mathbb{B}{\left(t(1-t)d(x,\partial X)\right)}^p \frac{dt}{t}dx.
\end{aligned}
\end{equation}
Therefore
$$
\begin{aligned}
&\int_{\mathbb{B}}{\left(t(1-t)d(x,\partial X)\right)}^p \left|\nabla_ \mathbb{B} {u}_{\epsilon}\right|^{p} {\xi}^{p}\frac{dt}{t}dx \\ \leq&-\int_{\mathbb{B}} {\left(t(1-t)d(x,\partial X)\right)}^p \left((-t^pf(t,x))_{\epsilon}+(n-p)|\nabla_ \mathbb{B}{u}_{\epsilon}|^{p-2}(t\partial_t {u}_{\epsilon})\right)\tilde{\psi}\frac{dt}{t}dx\\+& \int_{\mathbb{B}}p{\left(t(1-t)d(x,\partial X)\right)}^p\xi^{p-1}(\mathop{\sup }\limits_{\mathbb{B}}{u}_{\epsilon}-{u}_{\epsilon})\left|\nabla_ \mathbb{B} {u}_{\epsilon}\right|^{p-2}\nabla_ \mathbb{B} {u}_{\epsilon} \cdot \nabla_ \mathbb{B}\xi \frac{dt}{t}dx.
\end{aligned}
$$
By Young's inequality, we have

$$\begin{aligned}
&\int_{\mathbb{B}}{\left(t(1-t)d(x,\partial X)\right)}^p\left|\nabla_ \mathbb{B} {u}_{\epsilon}\right|^{p} {\xi}^{p}\frac{dt}{t}dx \\ \leq&\int_{\mathbb{B}} - {\left(t(1-t)d(x,\partial X)\right)}^p (-t^pf(t,x))_{\epsilon}\tilde{\psi}\frac{dt}{t}dx\\+&\theta \int_{\mathbb{B}}{\left(t(1-t)d(x,\partial X)\right)}^p\left|\nabla_ \mathbb{B} {u}_{\epsilon}\right|^{p} \xi^p\frac{dt}{t}dx \\+&C_{1}(\theta) \int_{\mathbb{B}}{\left(t(1-t)d(x,\partial X)\right)}^p\xi^{p}(\mathop{\sup }\limits_{\mathbb{B}}{u}_{\epsilon}-{u}_{\epsilon})^{p} \frac{dt}{t}dx \\+& \theta \int_{\mathbb{B}}{\left(t(1-t)d(x,\partial X)\right)}^p\left|\nabla_ \mathbb{B} {u}_{\epsilon}\right|^{p} \xi^p\frac{dt}{t}dx\\ +&C_{2}(\theta) \int_{\mathbb{B}}{\left(t(1-t)d(x,\partial X)\right)}^p(\mathop{\sup }\limits_{\mathbb{B}}{u}_{\epsilon}-{u}_{\epsilon})^{p} |\nabla_ \mathbb{B}\xi|^{p} \frac{dt}{t}dx.
\end{aligned}$$

Since  $ |\nabla \xi|\leq \frac{C}{r(\epsilon)}$ in $\mathbb{B}_{2r(\epsilon)} \setminus\mathbb{B}_{3r(\epsilon)} $,  and $|\nabla_{\mathbb{B}} \xi|\leq|\nabla \xi|$, then
$$\begin{aligned} &\int_{\mathbb{B}}{\left(t(1-t)d(x,\partial X)\right)}^p(\mathop{\sup }\limits_{\mathbb{B}}{u}_{\epsilon}-{u}_{\epsilon})^{p} |\nabla_ \mathbb{B}\xi|^{p} \frac{dt}{t}dx\\=&\int_{\mathbb{B}_{2r(\epsilon)} \setminus\mathbb{B}_{3r(\epsilon)}}{\left(t(1-t)d(x,\partial X)\right)}^p(\mathop{\sup }\limits_{\mathbb{B}}{u}_{\epsilon}-{u}_{\epsilon})^{p} |\nabla_ \mathbb{B}\xi|^{p} \frac{dt}{t}dx \\ \leq& ||u||_{L^{\infty}(\mathbb{B})}^p \left(\frac{{c_1}}{{r(\epsilon)}^p}\cdot {r(\epsilon)}^p\right)= {c_1}||u||_{L^{\infty}(\mathbb{B})}^p.
\end{aligned}$$
 Since $u\in L^{\infty}(\mathbb{B})$, and choosing $\theta=\frac{1}{4}$, then
 $$ \begin{aligned}
&\int_{\mathbb{B}}{\left(t(1-t)d(x,\partial X)\right)}^p\left|\nabla_ \mathbb{B} {u}_{\epsilon}\right|^{p} {\xi}^{p}\frac{dt}{t}dx \\ \leq&\int_{\mathbb{B}} - {\left(t(1-t)d(x,\partial X)\right)}^p (-t^pf(t,x))_{\epsilon}\tilde{\psi}\frac{dt}{t}dx\\ +&C_{1}(\theta) \int_{\mathbb{B}}{\left(t(1-t)d(x,\partial X)\right)}^p\xi^{p}(\mathop{\sup }\limits_{\mathbb{B}}{u}_{\epsilon}-{u}_{\epsilon})^{p} \frac{dt}{t}dx \\+&C_{2}(\theta) \int_{\mathbb{B}}{\left(t(1-t)d(x,\partial X)\right)}^p(\mathop{\sup }\limits_{\mathbb{B}}{u}_{\epsilon}-{u}_{\epsilon})^{p} |\nabla_ \mathbb{B}\xi|^{p} \frac{dt}{t}dx\\ \leq& C(\theta,p)\bigg(||u||_{{L}^{\infty}(\mathbb{B})}|| f||_{\mathbb{L}_{p}^{\gamma}}(\mathbb{B})+||u||_{L^{\infty}(\mathbb{B})}^p\int_{\mathbb{B}}{\left(t(1-t)d(x,\partial X)\right)}^p \frac{dt}{t}dx\\+&||u||_{L^{\infty}(\mathbb{B})}^p \bigg),
\end{aligned}$$

 so we have $$ \int_{\mathbb{B}_{3r(\epsilon)}}{\left(t(1-t)d(x,\partial X)\right)}^p\left|\nabla_ \mathbb{B} {u}_{\epsilon}\right|^{p} \frac{dt}{t}dx \leq C $$ where C is independent of $\epsilon$.

Since we find a uniform bound for $|\nabla_ \mathbb{B} {u}_{\epsilon}|$ in $\mathbb{L}_{p}^{\frac{n-p}{p}}(\mathbb{B}_{3r(\epsilon)})$, so $(\nabla_ \mathbb{B} {u}_{\epsilon})_{i}$  converges to $g_i(t,x)$ weakly in $\mathbb{L}_{p}^{\frac{n-p}{p}}(\mathbb{B})$, $i=1,2,\cdots, n$. Since ${u}_{\epsilon}$ converges pointwise to $u$, then for all $\psi\in\mathcal{C}_{0}^{\infty}(\mathbb{B})$, we have
$$ \begin{aligned}
\int_{\mathbb{B}}g_i(t,x){\psi}\frac{dt}{t}dx=\lim _{\epsilon\rightarrow 0}\int_{\mathbb{B}}(\nabla_ \mathbb{B} {u}_{\epsilon} )_i \psi\frac{dt}{t}dx=\int_{\mathbb{B}}{u}_{\epsilon} (\nabla_ \mathbb{B} \psi)_i\frac{dt}{t}dx=\int_{\mathbb{B}} {u} (\nabla_ \mathbb{B} \psi)_i\frac{dt}{t}dx
\end{aligned}
$$
then we derive $\nabla_ \mathbb{B} u=g(t,x)=\left(g_1(t,x),\cdots, g_n(t,x)\right) $, then we have
$$||u||_{ \mathbb{H}_{p}^{1,\frac{n-p}{p}}(\mathbb{B})}\leq C,$$
$$ ||u||_{ \mathbb{H}_{p,loc}^{1,\frac{n}{p}}(\mathbb{B})}< +\infty.$$
 Since $\nabla_ \mathbb{B} {u}_{\epsilon}$  converges to $\nabla_ \mathbb{B} {u}$ weakly in $\mathbb{L}_{p}^{\frac{n-p}{p}}(\mathbb{B})$ ,
 $$
\int_{\mathbb{B}}\theta( \left|\nabla_ \mathbb{B} {u}\right|^{p-2}\nabla_ \mathbb{B} {u}-\left|\nabla_ \mathbb{B}{u}_{\epsilon}\right|^{p-2}\nabla_ \mathbb{B} {u}_{\epsilon})\cdot \nabla_ \mathbb{B}({u}-{u}_{\epsilon}) \frac{dt}{t}dx \to 0
$$

 \begin{equation}\label{2.88}
 \int_{K}( \left|\nabla_ \mathbb{B} {u}\right|^{p-2}\nabla_ \mathbb{B} {u}-\left|\nabla_ \mathbb{B} {u}_{\epsilon}\right|^{p-2}\nabla_ \mathbb{B} {u}_{\epsilon})\cdot \nabla_ \mathbb{B}({u}-{u}_{\epsilon})
 \frac{dt}{t}dx\to 0.\end{equation}

 Though [Lemma 3.73], we know \eqref{2.88} implies $\nabla_ \mathbb{B}{u}_{\epsilon}\to\nabla_\mathbb{B}{u}$ for a.e. $(t,x)\in K$ and $$||\nabla_ \mathbb{B}{u}_{\epsilon}||_{\mathbb{L}^{\frac{n}{p}}_{p}( K)}\leq M<+\infty.$$ Finally taking  a limit on inequality \eqref{1.99} we have
 \begin{equation}\label{1.22}
\int_{\mathbb{B}}\left|\nabla_ \mathbb{B} u\right|^{p-2}\nabla_ \mathbb{B} u\cdot \nabla_ \mathbb{B}\psi \frac{dt}{t}dx \geq\int_{\mathbb{B}}  \left(-t^pf(t,x)+(n-p)|\nabla_ \mathbb{B}u|^{p-2}(t\partial_t u)\right)\psi\frac{dt}{t}dx.
\end{equation}

 \end{proof}

\begin{proof}[\textbf{Proof of Theorem \ref{Theorem 1.6}}]
If $u$ is a viscosity subsolution to \eqref{eq:11}, then $-u$ is a viscosity supersolution to \eqref{eq:11} with $-f$ instead of $f$. So we have $-u$ is  a weak supersolution to \eqref{eq:11} with $-f$ instead of $f$ by Theorem \ref{Theorem 1.5}, then  $u$ is  a weak subsolution to \eqref{eq:11}.
 \end{proof}
\begin{corollary}
    Under Assumption \ref{assumption2.2}, let  $f \in C( \overline{\mathbb B})\cap L^{\infty}(\mathbb{B})$, then
     \eqref{H8} has at least one weak solution $v$ satisfying \eqref{T1.10}.
 \end{corollary}

\subsection*{Acknowledgments}
The authors are grateful to the referees for their careful reading and valuable comments.

\end{document}